\documentclass[letterpaper,12pt,reqno]{amsart}
\usepackage{amscd}
\usepackage{indentfirst}
\usepackage{graphics}
\numberwithin{equation}{section}
\usepackage[margin=2.9cm]{geometry}
\usepackage{epstopdf}
\usepackage{adjustbox}
\usepackage{quiver}

\usepackage{amsmath, amsthm, amsfonts, amssymb}
\usepackage{relsize}
\usepackage{bbm}
\usepackage{mathtools}
\usepackage{stmaryrd}
\usepackage{graphicx}
\usepackage{tikz-cd}
\usepackage{enumitem}

\definecolor{cof}{RGB}{219,144,71}
\definecolor{pur}{RGB}{186,146,162}
\definecolor{greeo}{RGB}{91,173,69}
\definecolor{greet}{RGB}{52,111,72}

\usepackage{enumitem}
\usepackage{hyperref}
\hypersetup{
    colorlinks=true,
    linkcolor=violet,
    filecolor=magenta,      
    urlcolor=blue,
    citecolor=red
}
\usetikzlibrary{decorations.markings,positioning}
\definecolor{pur}{RGB}{186,146,162}
\pgfdeclarelayer{background}
\pgfsetlayers{background,main}
\newtheorem{thm}{Theorem}[section]
\newtheorem*{thm*}{Theorem}
\newtheorem{prop}[thm]{Proposition}
\newtheorem{lem}[thm]{Lemma}
\newtheorem{cor}[thm]{Corollary}

\newtheorem{conj}[thm]{Conjecture}

\theoremstyle{definition}
\newtheorem{defn}[thm]{Definition}
\newtheorem{ex}[thm]{Example}
\newtheorem{rmk}[thm]{Remark}

\newcommand{\longhookrightarrow}{\lhook\joinrel\longrightarrow}

\newcommand{\tri}{\ol{\nabla}}

\newcommand{\bV}{\mathbb{V}}

\newcommand{\lifts}{\boxslash}

\newcommand{\cin}{\subseteq}

\newcommand\ol[1]{\ensuremath{\overline{#1}}}
\newcommand\abs[1]{\ensuremath{\lvert#1\rvert}}

\newcommand{\nospace}[1]{\makebox[0pt][l]{\,#1}}

\newcommand{\mcW}{\mathcal{W}}
\newcommand{\mcE}{\mathcal{E}}

\newcommand{\mcT}{\mathcal{T}}

\newcommand{\mcU}{\mathcal{U}}

\newcommand{\mcA}{\mathcal{A}}

\newcommand{\mcV}{\mathcal{V}}

\newcommand{\op}{\operatorname{op}}
\newcommand{\id}{\operatorname{id}}
\newcommand{\sk}{\operatorname{sk}}

\newcommand{\Hom}{\operatorname{Hom}}

\newcommand{\SpHorn}{\operatorname{SpHorn}}
\newcommand{\SpOutHorn}{\operatorname{SpOH}}

\newcommand{\InnHorn}{\operatorname{IH}}

\newcommand{\Set}{\operatorname{Set}}
\newcommand{\sSet}{\operatorname{sSet}}

\newcommand{\Cat}{\operatorname{Cat}}

\newcommand{\Min}{\operatorname{Min}}

\begin{document}

\title{Generalizing quasi-categories via model structures on simplicial sets}

\author[M. Feller]{Matt Feller}

\address{Department of Mathematics, University of Virginia, Charlottesville, VA 22904}

\email{feller@virginia.edu}

\date{\today}

\thanks{The author was partially supported by NSF RTG grant DMS-1839968 and NSF grant DMS-1906281.}

\begin{abstract}
    We use Cisinski's machinery to construct and study model structures on the category of simplicial sets whose classes of fibrant objects generalize quasi-categories. We identify a lifting condition that captures the homotopical behavior of quasi-categories without the algebraic aspects and show that there is a model structure whose fibrant objects are precisely those that satisfy this condition. We also identify a localization of this model structure whose fibrant objects satisfy a ``special horn lifting'' property similar to the one satisfied by quasi-categories. This special horn model structure leads to a conjectural characterization of the bijective-on-0-simplices trivial cofibrations of the Joyal model structure. We also discuss how these model structures all relate to one another and to the minimal model structure.
\end{abstract}

\maketitle

\setcounter{tocdepth}{1}

\tableofcontents

\section{Introduction}

The theory of quasi-categories has proven to be a powerful tool across many areas of mathematics, including algebraic geometry, topology, and beyond. The basic idea of a quasi-category is that it is a simplicial set that behaves like a category ``up to homotopy.'' This paper explores how one can generalize this idea, where we have simplicial sets modeling up-to-homotopy versions of structures that are weaker than categories. Our motivating example of such a structure weaker than categories is the 2-Segal sets of Dyckerhoff-Kapranov \cite{DK} and G\'alvez-Kock-Tonks \cite{GKT:partI}. In a follow-up paper, we define ``quasi-2-Segal sets'' that are up-to-homotopy versions of 2-Segal sets, building on the groundwork laid here.

A foundational result of quasi-category theory is the existence of a corresponding model structure on the category of simplicial sets, called the Joyal model structure \cite{Joyal:theory}. A desirable quality of any generalization of quasi-categories is therefore the existence of a similar associated model structure. Taking this idea to heart, one could say that within simplicial sets the search for robust generalizations of quasi-categories is equivalent to the search for model structures. Hence, the aim of this paper is to dive into the sea of possible model structures and retrieve a few with properties that should prove useful for further study.

\subsection{Model structures on simplicial sets}

The two most prominent model structures on $\sSet$, the category of simplicial sets, are the Kan-Quillen model structure \cite{Quillen} and the Joyal model structure \cite{Joyal:theory}. In both model structures, all objects are cofibrant, so the well-behaved objects are precisely the fibrant objects. In the Kan-Quillen model structure, the fibrant objects are the Kan complexes which provide a model of spaces/$\infty$-groupoids, and the fibrant objects of the Joyal model structure are the quasi-categories which give us a model of $(\infty,1)$-categories. 
These model structures are both examples of \emph{Cisinski model structures} on $\sSet$, meaning that they are cofibrantly generated and their cofibrations are precisely the monomorphisms. The Kan-Quillen model structure is a localization of the Joyal model structure in the sense that it has the same cofibrations and its class of fibrant objects (Kan complexes) is contained in the class of Joyal fibrant objects (quasi-categories). In general, the process of localizing a model structure to another with the same cofibrations and fewer fibrant objects is well understood; see \cite{Hirschhorn}. There are other localizations of the Joyal model structure in the literature, for example in \cite{CL} and in Chapter 9 of \cite{Cisinski:Asterisque}.

By starting with the Joyal model structure and localizing, one ends up with fibrant objects that are quasi-categories with extra structure. If we instead want to do the opposite and generalize the notion of quasi-category, then we want to ``de-localize.'' The goal of this paper is to lay the groundwork for constructing such de-localizations of the Joyal model structure. Our approach is to focus on the homotopical aspects of quasi-categories, constructing various model structures that maintain those aspects but lack a notion of composition. In particular, for morphisms $f$ and $g$ in a quasi-category $Q$, we consider a homotopy from $f$ to $g$ to be given by a 2-simplex
\[
\adjustbox{scale=0.9}{
\begin{tikzcd}[row sep=small]
	x && y \\
	& x
	\arrow["{s_0 x}"', dashed, from=1-1, to=2-2]
	\arrow["f", from=1-1, to=1-3]
	\arrow["g"', from=2-2, to=1-3]
\end{tikzcd}
}
\]
with degenerate edge $0\to 1$ as indicated. We say that $Q$ is \emph{homotopically-behaved}, in the sense that all of the higher invertibility and compositionality we would expect from a good notion of homotopy are satisfied by the 2-simplices of this form, as well as by higher $n$-simplices whose edge $i\to i+1$ for some $0\leq i\leq n-1$ is degenerate. The purpose of this paper is to study Cisinski model structures on $\sSet$ whose fibrant objects are homotopically-behaved, which we call \emph{homotopically-behaved model structures}. Our main result is to construct and describe the homotopically-behaved model structure with the smallest possible class of weak equivalences.

\begin{thm*}\label{thm:main}
There exists a minimal homotopically-behaved model structure on $\sSet$. The fibrant objects in this model structure are the simplicial sets with lifts of certain modified horn inclusions, which we call \emph{$J$-augmented horn inclusions}.
\end{thm*}

We state this theorem in more detail as Theorem \ref{thm:summarythm}. The terms \emph{homotopically-behaved} and \emph{$J$-augmented horn inclusion} are defined explicitly in Subsection \ref{sub:behavedandaugmented}.

\subsection{2-Segal motivation}

Our motivation for considering de-localizations of the Joyal model structure is to construct a ``quasi-2-Segal set'' model structure, where the fibrant objects satisfy an up-to-homotopy version of the the 2-Segal condition introduced in \cite{DK} and \cite{GKT:partI}. Recall that the (strict) Segal (or ``1-Segal'') condition encodes unique composition; the simplicial sets satisfying this condition (the ``1-Segal sets'') are equivalent to categories. The simplicial sets satisfying a weakened, up-to-homotopy version of the 1-Segal condition are the quasi-categories. The 2-Segal condition is generalization of the 1-Segal condition that encodes partially-defined, not-necessarily-unique composition which is still associative in a particular sense. It is therefore natural to try to extend this generalization from 1-Segal to 2-Segal to the up-to-homotopy setting, i.e., to look for a robust definition of ``quasi-2-Segal sets.'' A compelling justification for a particular definition would be the existence of a model structure analogous to the Joyal model structure. Since the quasi-2-Segal sets should generalize quasi-categories, such a model structure should be a de-localization of the Joyal model structure. In follow-up work \cite{Feller:q2} we construct such a quasi-2-Segal model structure by localizing our minimal homotopically-behaved model structure with respect to maps that encode the 2-Segal condition.

\subsection{Cisinski's theory and the minimal model structure}

Proving the existence of a model structure from scratch is generally cumbersome and highly technical, but fortunately for our particular situation Cisinski's theory provides a powerful framework for building model structures that requires checking a much more manageable set of conditions. One aspect of this theory is the existence of a minimal model structure, whose class of fibrant objects contains the fibrant objects of every other Cisinski model structure.\footnote{One could also call this ``maximal,'' but we prefer ``minimal'' since the class of weak equivalences is as small as possible, and the fibrant objects have the least structure.} Therefore, one approach to de-localizing the Joyal model structure is to de-localize all the way back to the minimal model structure, and then localize from there.

As we see in a companion paper \cite{Feller:minimal}, we lose a lot by de-localizing all the way down to the minimal model structure. In particular, the main result of that paper is a new characterization of the fibrant objects in the minimal model structure. What we find is that the notion of ``homotopy'' familiar from quasi-categories does not behave well in the fibrant objects of this model structure.
In particular, the existence of a homotopy from $f$ to $g$ given by a 2-simplex with degenerate edge need not imply the existence a homotopy from $g$ to $f$. Homotopies of this form also need not compose.
One can devise an alternative notion of homotopy which behaves well in the minimal model structure, essentially by demanding that all higher invertibility data be present from the start, but then one sacrifices the simplicity that comes from a homotopy being embodied by a single simplex.

\subsection{Homotopically-behaved model structures and augmented horns}\label{sub:behavedandaugmented}

We denote by $\Delta[n]^{\ast}_{i\to i+1}$ the standard $n$-simplex with the edge $i\to (i+1)$ collapsed to a degeneracy. The idea that an $n$-simplex with degenerate $i\to (i+1)$ edge is a homotopy of $(n-1)$-simplices can be expressed by the surjective map $\Delta[n]^{\ast}_{i\to i+1}\to \Delta[n-1]$ being a weak equivalence. Thus, we introduce the terminology \emph{homotopically-behaved} for Cisinski model structures on $\sSet$ where each of these surjective maps $\Delta[n]^{\ast}_{i\to i+1}\to \Delta[n-1]$ is a weak equivalence. In Section \ref{sec:behaved} we construct the minimal homotopically-behaved model structure, whose fibrant objects have the least possible structure while maintaining the desirable homotopical aspects of quasi-categories. Localizing this model structure with respect to the maps from the 2-Segal condition yields a ``quasi-2-Segal set'' model structure whose fibrant objects must then also have the desirable homotopical aspects of quasi-categories.

At the same time, we construct a nontrivial localization of this model structure at $K\to \ast$, where $K$ is the simplicial set one gets by gluing in a left and right inverse to $\Delta[1]$; see Example \ref{ex:K}. Although this model structure does not appear to be directly useful for defining quasi-2-Segal sets, it may be of independent interest in understanding the broader picture of model structures on simplicial sets.

The key inspiration for our approach comes from the special outer horn extension property of quasi-categories. Recall that a \emph{horn} $\Lambda^i[n]$ is the union of all of the faces of the $n$-simplex $\Delta[n]$ except for $d_i\Delta[n]$, which we say is \emph{inner} if $0<i<n$ and is \emph{outer} if $i=0$ or $i=n$. We refer to these horns as \emph{ordinary horns} to distinguish them from the augmented horns we introduce below. A \emph{quasi-category} is a simplicial set $Q$ such that every inner horn $\Lambda^i[n]\to Q$ extends to an $n$-simplex $\Delta[n]\to Q$. A quasi-category need not have extensions of an outer horn such as $\Lambda^0[n]\to Q$. However, if the edge $0\to 1$ is sent to an edge in $Q$ that is invertible in a certain sense, then we say $\Lambda^0[n]\to Q$ is an example of a \emph{special outer horn}, and it turns out that quasi-categories do have extensions of all special outer horns. 

In a general homotopically-behaved model structure, the fibrant objects need not have extensions of ordinary horns. However, the central idea of our approach is to create augmented horns, where we glue a simplicial set onto a particular edge to ``invert'' it. In Section \ref{sec:augmentation}, we see how certain augmented horn inclusions are forced to be weak equivalences in a homotopically-behaved model structure. Furthermore, we can characterize the fibrant objects in the minimal homotopically-behaved model structure in terms of lifts of $J$-augmented horn inclusions, as we stated in Theorem \ref{thm:main}. We denote by $J$ the nerve of the free-living isomorphism, so and define \emph{$J$-augmented horn inclusions} to be ordinary horn inclusions $\Lambda^j[n]\hookrightarrow \Delta[n]$ with a copy of $J$ glued in along either the $(j-1)\to j$ edge or the $j\to (j+1)$ edge. We see in Corollary \ref{cor:KbehavedlocalizestoJoyal} that the minimal homotopically-behaved model structure localizes to the Joyal model structure.

\subsection{Special horns}

The notion of invertibility in Section \ref{sec:augmentation} comes from attaching a simplicial set $I$ that is contractible in the sense that the map $I\to \ast$ is a weak equivalence, with our main examples being $I=J$ and $I=K$. In Section \ref{sec:special}, we see that this notion does not account for all of the edges we want to consider invertible in the context of quasi-categories, and identify a separate class of augmented horn inclusions which we call the \emph{special horn inclusions}. We show that there is a model structure whose fibrant objects are precisely the simplicial sets with special horn inclusions.

Recall that we deemed the localization of the minimal homotopically-behaved model structure at $K\to \ast$ to be unsuitable for our ultimate purposes in defining quasi-2-Segal sets. This special horn model structure is a further localization, and hence is also not suitable. However, this model structure may be of interest with regards to studying the Joyal model structure. In particular, we conjecture that the special horn inclusions, together with the inner horn inclusions, generate the class of bijective-on-0-simplices trivial cofibrations in the Joyal model structure; see Conjecture \ref{conj:joyalbijon0simp}.

\subsection{Pointwise cylinders}

A central element of Cisinski's theory is the \emph{exact cylinder}, which is a functorial choice of simplicial set $E\otimes X$ for each simplicial set $X$, satisfying certain axioms. Many applications of Cisinski's work use an exact cylinder given by the Cartesian product $E\otimes X=I\times X$ for some simplicial set $I$ with distinct vertices $\{0\}\hookrightarrow I$ and $\{1\}\hookrightarrow I$. For our purposes, the necessary proofs are greatly simplified by instead using an alternative kind of exact cylinder, which we call a \emph{pointwise cylinder}. We introduce pointwise cylinders in Subsection \ref{sub:pointwise}. In Subsection \ref{sub:augMS}, we see that the minimal homotopically-behaved model structure is also minimal with respect to pointwise cylinders, in the sense that we construct it using a pointwise cylinder in Cisinski's machinery and any other such constructed model structure is a localization of it.

\subsection{Organization}

In Section \ref{sec:background}, we cover basic definitions and notation, and then summarize Cisinski's theory. In Section \ref{sec:augmentation}, we define homotopically-behaved model structures and augmented horn inclusions. In Section \ref{sec:behaved} we show that there is a minimal homotopically-behaved model structure and that its fibrant objects are the simplicial sets with extensions of certain augmented horns. In Section \ref{sec:special}, we define \emph{special horns} as a separate kind of augmented horn, and use Cisinski's machinery to show that there exists a model structure whose fibrant objects are simplicial sets with special horn extensions. In Section \ref{sec:comparison} we summarize and compare the various model structures constructed in this work.

\subsection{Acknowledgement}

I would like to thank Julie Bergner for her extraordinarily helpful feedback.

\section{Background}\label{sec:background}

We recall some basic notions, as well as the necessary aspects of Cisinski's theory.

\subsection{Basics of simplicial sets and model structures}

Let $\Delta$ denote the category whose objects are the finite non-empty ordered sets $[n]=\{0\leq 1\leq \ldots\leq n\}$ for $n\geq 0$ and whose morphisms are order-preserving maps. We write $d^i\colon [n]\to [n+1]$ and $s^i\colon [n+1]\to [n]$ for the co-face and co-degeneracy maps, respectively, which generate the morphisms of $\Delta$. A \emph{simplicial set} is a functor $\Delta^{\op}\to \Set$. We denote the category of simplicial sets by $\sSet$ and the representable simplicial sets by $\Delta[n]$, except that we often denote $\Delta[0]$ instead by $\ast$ since it is the terminal object in $\sSet$. 
Our notation for the $i$th horn of $\Delta[n]$ is $\Lambda^i[n]$. For more background on simplicial sets, see \cite{GJ}.
We write $f\lifts g$ if $g$ has the right lifting property with respect to $f$.
The class of morphisms with the right lifting property with respect to a set of maps $\mathcal{A}$ is denoted by $\mathcal{A}^{\lifts}$, and the class of morphisms $f$ such that $f\lifts \mathcal{B}$ is denoted by $^{\lifts}\mathcal{B}$. Given a set $S$ of morphisms, the class $^{\lifts}(S^{\lifts})$ is the closure of $S$ under taking pushouts, transfinite compositions, and retracts. For this reason, we sometimes say that $S$ \emph{generates} the class $^{\lifts}(S^{\lifts})$.


We restrict our focus to \emph{Cisinski} model structures on $\sSet$, which are cofibrantly generated model structures whose cofibrations are precisely the monomorphisms. A model structure is \emph{cofibrantly generated} if there are sets $\mathcal{I}$ and $\mathcal{J}$ such that $\mathcal{I}$ generates the cofibrations and $\mathcal{J}$ generates the trivial cofibrations.
For more background on model categories, see \cite{Hirschhorn} or \cite{Hovey}.

\subsection{Cisinski's theory}

The main result we use from Cisinski is Theorem \ref{thm:cisinskiMS}, which says that if a set of monomorphisms $\Lambda$ satisfies certain properties, then $\Lambda$ characterizes the fibrant objects of some model category via a lifting property. The purpose of this subsection is to recall the background necessary to state this result precisely.

We begin by recalling the notion of a cylinder.

\begin{defn} \cite[Def.~2.4.6]{Cisinski:Cambridge}
A \emph{cylinder} of a simplicial set $X$ is a factorization
\[
\begin{tikzcd}
X\sqcup X \arrow[r, "{(\partial_0,\partial_1)}", hook] & I\otimes X \arrow[r] & X
\end{tikzcd}
\]
of the canonical fold map $(\id_X, \id_X)$, where the first map is a monomorphism. The maps $\partial_{\varepsilon}$ for $\varepsilon=0,1$ pick out copies of $X$ that do not intersect inside of $I\otimes X$.
\end{defn}

Beware that the notation $I\otimes X$ in the above definition is purely formal. The simplicial set $I\otimes X$ need not be a monoidal product or tensor of any kind.

\begin{rmk}
In an arbitrary model category, we define cylinders similarly, where the first map is required to be a cofibration. In the context of Cisinski's theory, all of the model structures we consider have as cofibrations precisely the class of monomorphisms, justifying this definition.
\end{rmk}

A \emph{functorial cylinder} is a compatible choice of cylinder for every simplicial set. To make this definition rigorous, we first notice that $X\mapsto X\sqcup X$ and $X\mapsto X$ are endofunctors of $\sSet$, which we denote by $1\sqcup 1$ and $1$, respectively. There is a natural transformation $(\id,\id)\colon 1\sqcup 1\Longrightarrow 1$ whose component at each $X$ is the canonical fold map $(\id_X,\id_X)$.

\begin{defn}\cite[Def.~2.4.8]{Cisinski:Cambridge}
A \emph{functorial cylinder} is a factorization
\[
\begin{tikzcd}
1\sqcup 1 \arrow[r, "{(\partial_0,\partial_1)}", Rightarrow] & I\otimes - \arrow[r, Rightarrow] & 1
\end{tikzcd}
\]
of the natural transformation $(\id, \id)$, where each component of the natural transformation $(\partial_0,\partial_1)\colon 1\sqcup 1\Longrightarrow I\otimes -$ is a monomorphism.
\end{defn}

The motivation behind these definitions is to generalize the idea (from the Kan-Quillen model structure on $\sSet$) of $\Delta[1]\times X$ being a cylinder of a simplicial set $X$, in the sense that a map $X\times \Delta[1]\to Y$ gives a homotopy of maps $X\to Y$. Imposing some further conditions helps maintain the spirit of the original setting, where we imagine $I\otimes X$ as something like a stretched out copy of $X$.

\begin{defn}\label{def:exactcylinder}\cite[Def.~2.4.8]{Cisinski:Cambridge}
An \emph{exact cylinder}\footnote{The origin of this definition is \cite{Cisinski:Asterisque}, where the term is ``\emph{donn\'ee homotopique \'el\'ementaire},'' hence ``DH.''} is a functorial cylinder satisfying the following axioms.
\begin{enumerate}[start=1,label={(DH\arabic*).\ }, widest=(An2$'$.), leftmargin=*]
\item The functor $I\otimes -$ commutes with small colimits and preserves monomorphisms.
\item For any monomorphism of simplicial sets $j\colon A\hookrightarrow B$, the square
\[
\begin{tikzcd}
A \arrow[r, "j", hook] \arrow[d, "(\partial_{\varepsilon})_A"'] & B \arrow[d, "(\partial_{\varepsilon})_B"] \\
I\otimes A \arrow[r, "I\otimes j"']                             & I\otimes B                               
\end{tikzcd}
\]
is a pullback for each $\varepsilon = 0,1$.
\end{enumerate}
\end{defn}

\begin{rmk}\label{rmk:exactpullback}
Assuming $I\otimes -$ preserves monomorphisms (as (DH1) calls for), all of the morphisms in the square in (DH2) are monomorphisms. We can interpret the condition that the square be a pullback as saying that the intersection of $I\otimes A\cin I\otimes B$ with $B\cin I\otimes B$ is precisely $A$.
\end{rmk}

\begin{ex}\label{ex:exact}
Let $I$ be a cylinder of the terminal simplicial set $\Delta[0]$. In other words, choose a monomorphism of simplicial sets $\Delta[0]\sqcup \Delta[0]\hookrightarrow I$. The functor defined by $I\otimes X= I\times X$ determines an exact cylinder. In particular, when $I=\Delta[1]$, we recover the familiar notion of cylinder from the Kan-Quillen model structure.

Many applications of this theory involve cylinders defined by Cartesian product, but our approach uses a new kind of exact cylinder which we introduce in Subsection \ref{sub:pointwise}.
\end{ex}

Let us follow Remark 2.4.9 in \cite{Cisinski:Cambridge} and establish some more notation. Given an exact cylinder $I\otimes -$, by taking the pushout of the left and top maps in the square for axiom (DH2), we get the inclusion map
\[
(I\otimes A) \cup B \longhookrightarrow I\otimes B.
\]
To emphasize the dependence of this inclusion on whether $\varepsilon=0$ or 1 in the inclusion $\partial_{\varepsilon}\colon \Delta[0]\hookrightarrow I$, we rewrite this map as
\[
(I\otimes A) \cup (\{\varepsilon\} \otimes B) \longhookrightarrow I\otimes B,
\]
thinking of $\{0\}$ and $\{1\}$ as endpoints of $I=I\otimes \Delta[0]$.

In this spirit, we also write $\partial I$ for the union of $\{0\}$ and $\{1\}$ inside of $I$, and we write $\partial I\otimes X$ as the union of $\{0\}\otimes X$ and $\{1\}\otimes X$ in $I\otimes X$. Then we have a canonical inclusion
\[
(I\otimes A) \cup (\partial I \otimes B) \longhookrightarrow I\otimes B,
\]
arising from a diagram akin to the square in axiom (DH2).

We are now ready for one of the key definitions we need from Cisinski.

\begin{defn}\cite[Def.~2.4.11]{Cisinski:Cambridge}\label{def:anodyneclass}
Given an exact cylinder $I\otimes -$, we say that a class of morphisms ${}^\lifts(\Lambda^\lifts)$ generated by a set $\Lambda$ of monomorphisms is an \emph{$(I\otimes -)$-anodyne class} if the following conditions hold.
\begin{enumerate}[start=1,label={(An\arabic*).\ }, widest=(An2$'$.), leftmargin=*]
\item For each monomorphism of simplicial sets $X\hookrightarrow Y$ and $\varepsilon=0,1$, the induced map $(I\otimes X)\cup (\{\varepsilon\}\otimes Y)\hookrightarrow I\otimes Y$ is in ${}^\lifts(\Lambda^\lifts)$.
\item For each $A\hookrightarrow B$ in ${}^\lifts(\Lambda^\lifts)$, the induced map $(I\otimes A)\cup (\partial I \otimes B)\to I\otimes B$ is also in ${}^\lifts(\Lambda^\lifts)$.
\end{enumerate}
\end{defn}

We can restate each of axioms (An1) and (An2) in a form that is easier to check.

\begin{lem}\label{lem:altaxioms}
Let $I\otimes-$ be an exact cylinder and let $\Lambda$ be a set of monomorphisms. Then axiom \emph{(An1)} is equivalent to \emph{(An1$'$)} below and axiom \emph{(An2)} is equivalent to \emph{(An2$'$)} below.
\end{lem}
\begin{enumerate}[start=1,label={(An\arabic*$'$).\ }, widest=(An2$'$.), leftmargin=*]
\item For each $n\geq 0$ and $\varepsilon=0,1$, the map $(I\otimes \partial \Delta[n])\cup (\{\varepsilon\}\otimes \Delta[n])\hookrightarrow I\otimes \Delta[n]$ induced by $\partial\Delta[n]\hookrightarrow \Delta[n]$ is in ${}^\lifts(\Lambda^\lifts)$.
\item For each $A\hookrightarrow B$ in $\Lambda$, the induced map $(I\otimes A)\cup (\partial I \otimes B)\to I\otimes B$ is in ${}^\lifts(\Lambda^\lifts)$.
\end{enumerate}
\begin{proof}
The equivalence (An1$'$)$\iff$(An1) follows from the equality of classes
\[
\begin{split}
&\{(I\otimes \partial\Delta[n])\cup (\{\varepsilon\}\otimes \Delta[n])\hookrightarrow I\otimes \Delta[n]\mid n\geq 0, \ \varepsilon=0,1\}^{\lifts}\\
=\ &\{(I\otimes X)\cup (\{\varepsilon\}\otimes Y)\hookrightarrow I\otimes Y\mid X\hookrightarrow Y\text{ in }\sSet,\ \varepsilon=0,1\}^{\lifts},
\end{split}
\]
which is a consequence of correspondence (2.4.13.4) of Example 2.4.13 in \cite{Cisinski:Cambridge} which ultimately relies on the fact that the boundary inclusions generate the class of monomorphisms. The equivalence of (An2$'$) and (An2) follows from a similar argument, replacing $\{\varepsilon\}$ with $\partial I$ and arbitrary monomorphisms $X\hookrightarrow Y$ with maps in ${}^\lifts(\Lambda^\lifts)$.
\end{proof}

\begin{defn}
Given an exact cylinder $I\otimes -$ and a morphism of simplicial sets $f_0,f_1\colon A\to X$, we say that an \emph{$I$-homotopy} from $f_0$ to $f_1$ is a map $h\colon I\otimes A\to X$ such that precomposing $h$ with $\{\varepsilon\}\otimes A\hookrightarrow I\otimes A$ yields $f_{\varepsilon}$ for $\varepsilon=0,1$. We say that $f,g\colon A\to X$ are \emph{$I$-homotopic} if there is a finite zigzag of $I$-homotopies from $f$ to $g$. Suppressing the dependence on $I$, we let $[A,X]$ denote the quotient of the set $\Hom(A,X)$ by identifying maps that are $I$-homotopic.
\end{defn}

\begin{thm}\label{thm:cisinskiMS}\cite[Thm.~2.4.19]{Cisinski:Cambridge}
Given an exact cylinder $I\otimes -$ and a set of monomorphisms $\Lambda$ such that ${}^\lifts(\Lambda^\lifts)$ is an $(I\otimes-)$-anodyne class, there is a cofibrantly generated model structure on $\sSet$ whose cofibrations are the monomorphisms and whose fibrant objects are the simplicial sets with the right lifting property with respect to $\Lambda$. The weak equivalences in this model structure are maps $X\to Y$ such that for every fibrant $W$ the induced map $[Y,W]\to [X,W]$ is a bijection.
\end{thm}

\begin{rmk}\label{rmk:MSdeterminedbyfibobj}
It is a theorem of Joyal that if two model structures share the same cofibrations and fibrant objects, then they are the same model structure; see Proposition E.1.10 in \cite{Joyal:theory}. Therefore, the weak equivalences described in the above theorem are determined once we know our cofibrations are the monomorphisms and our fibrant objects are those with lifts against $\Lambda$. Throughout this paper, we shall implicitly use this fact from Joyal to conclude that when the class of fibrant objects of a Cisinski model structure is contained in the class of fibrant objects in another, the class of weak equivalences of the former model contains the class of weak equivalences of the latter.
\end{rmk}

\section{Homotopically-behaved model structures and augmented horns}\label{sec:augmentation}

In this section, we define \emph{homotopically-behaved} model structures to be Cisinski model structures where the retract map from an $n$-simplex with $i\to (i+1)$ edge collapsed to a degeneracy onto the $(n-1)$-simplex is a weak equivalence. This condition captures the idea that a map out of an $n$-simplex with degenerate $i\to (i+1)$ edge is a homotopy of $(n-1)$-simplices. We show that this condition is equivalent to the condition that certain modified horn inclusions are weak equivalences.

\subsection{Augmented horn extensions}

In a quasi-category, we can view higher homotopies as simplices with a degenerate edge $i\to i+1$. More precisely, for $n\geq 1$, a homotopy between $n$-simplices $x,y$ consists of an $(n+1)$-simplex $H$ with some $0\leq i\leq n$ such that the edge $i\to (i+1)$ is degenerate and $\{d_i H, d_{i+1} H\}=\{x,y\}$.

Before continuing this discussion, let us fix some notation.
\begin{defn}
Given $n\geq 2$, $0\leq i\leq n-1$, a subcomplex $A\cin \Delta[n]$ containing the $i\to (i+1)$ edge of $\Delta[n]$, and a map $\Delta[1]\to X$, let $A^{X}_{i\to i+1}$ be the pushout
\[
\begin{tikzcd}
{\Delta[1]} \arrow[d, "i\to i+1"'] \arrow[r] & X \arrow[d]      \\
A \arrow[r]                                  & A^{X}_{i\to i+1}\nospace{.}
\end{tikzcd}
\]
\end{defn}

\begin{ex}
To get $\Lambda^1[2]^{X}_{1\to 2}$, we attach $\Delta[1]\to X$ along the edge $1\to 2$, as in the following diagram
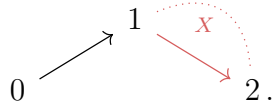
\[
\begin{tikzcd}[row sep=small]
	& 1 \\
	0 && 2\nospace{.}
	\arrow[from=2-1, to=1-2]
	\arrow[color={rgb,255:red,214;green,92;blue,92}, from=1-2, to=2-3]
	\arrow["X"', color={rgb,255:red,214;green,92;blue,92}, curve={height=-25pt}, dotted, no head, from=1-2, to=2-3]
\end{tikzcd}
\]
\end{ex}

Our main use for these $A^X_{i\to i+1}$ is to discuss modified versions of the ordinary horn inclusions. Given a horn inclusion $\Lambda^j[n]\hookrightarrow \Delta[n]$, the horn $\Lambda^j[n]$ contains the edges $(j-1)\to j$ and $j\to (j+1)$ (excluding the cases $-1\to 0$ and $n\to n+1$ as there are no such edges). For each $\Delta[1]\to X$ and $i=j-1, j$, we get an induced inclusion $\Lambda^j[n]^{X}_{i\to i+1}\hookrightarrow \Delta[n]^{X}_{i\to i+1}$. The modified horn inclusions we study are two special cases of this situation. The first case is when $X=\ast$, so the induced inclusion $\Lambda^j[n]^{\ast}_{i\to i+1}\hookrightarrow \Delta[n]^{\ast}_{i\to i+1}$ is the ordinary horn inclusion with the $i\to i+1$ edge collapsed to a degeneracy. The second case is when $\Delta[1]\to X$ is an inclusion, so that the induced map is the ordinary horn inclusion with a copy of $X$ glued in along the $i\to i+1$ edge. Let us set some terminology for these special cases.

\begin{defn}\label{def:pinchedandaugmentedhorns} Fix $n\geq 2$ and $0\leq i\leq n-1$.

\begin{enumerate}
    \item When $X=\ast$, the terminal simplicial set, we say that $\Lambda^j[n]^{\ast}_{i\to i+1}$ is a \emph{pinched horn} for $j=i, i+1$, that $\Delta[n]^{\ast}_{i\to i+1}$ is a \emph{pinched $n$-simplex}, and that the inclusion $\Lambda^j[n]^{\ast}_{i\to i+1}\hookrightarrow \Delta[n]^{\ast}_{i\to i+1}$ for $j=i, i+1$ is a \emph{pinched horn inclusion}.
    \item When $\Delta[1]\to X$ is an inclusion, we say that $\Lambda^j[n]^{X}_{i\to i+1}$ is an \emph{$X$-augmented horn} for $j=i, i+1$, that $\Delta[n]^{X}_{i\to i+1}$ is an \emph{$X$-augmented $n$-simplex}, and that the inclusion $\Lambda^j[n]^{X}_{i\to i+1}\hookrightarrow \Delta[n]^{X}_{i\to i+1}$ for $j=i, i+1$ is an \emph{$X$-augmented horn inclusion}.
\end{enumerate}
\end{defn}

\begin{rmk}
The notation $A^{X}_{i\to i+1}$ and terminology ``$X$-augmented'' are technically ambiguous because they depend on the map $\Delta[1]\to X$, but the choice of map should be clear from context in all instances in this paper.
\end{rmk}

Since we think of maps out of a pinched $(n+1)$-simplex $\Delta[n+1]^{\ast}_{i\to i+1}$ as homotopies of $n$-simplices, we can interpret this situation as saying that maps out of a pinched $(n+1)$-simplex $\Delta[n+1]^{\ast}_{i\to i+1}$ are equivalent to maps out of the standard $n$ simplex $\Delta[n]$ up to homotopy. This interpretation is encoded more precisely by the surjective map $\Delta[n+1]^{\ast}_{i\to i+1}\to \Delta[n]$ being a weak equivalence in the Joyal model structure. As our goal is to construct model structures with fibrant objects that have the same homotopical behavior as quasi-categories, a natural starting place is to declare each of these surjective maps $\Delta[n+1]^{\ast}_{i\to i+1}\to \Delta[n]$ for $n\geq 1$ and $0\leq i\leq n$ to be weak equivalences in our model structure. The aim of this subsection is to use 2-out-of-3 arguments to identify other maps that are forced to be weak equivalences in this situation.

\begin{defn}
We say that a Cisinski model structure on the category of simplicial sets is a \emph{homotopically-behaved model structure} if the surjective maps $\Delta[m+1]^{\ast}_{k\to k+1}\to \Delta[m]$ are weak equivalences for all $m\geq 1$ and $0\leq k \leq m$.
\end{defn}

Our first step is to notice that these maps have sections, which must also be weak equivalences by the 2-out-of-3 property.

\begin{lem}\label{lem:sectionsweakequiv}
The surjective maps $\Delta[m+1]^{\ast}_{k\to k+1}\to \Delta[m]$ are weak equivalences for all $m\geq 1$ and $0\leq k \leq m$ if and only if the sections $d^{k},d^{k+1}\colon\Delta[m]\hookrightarrow \Delta[m+1]^{\ast}_{k\to k+1}$ are as well.
\end{lem}

Our next step is to show that pinched horn inclusions are forced to be weak equivalences. Our proof requires first defining \emph{generalized pinched horns}. Given $n\geq 2$, $0\leq i\leq n-1$, and a subset $S\cin \{0,\ldots,n\}$ such that $\abs{S}\geq 2$ and exactly one of $i,i+1$ is in $S$, the generalized horn $\Lambda^{S}[n]\cin \Delta[n]$ is the union of all $d_j$ faces of $\Delta[n]$ for $j$ in $S$. There exists an $\ell\in S$ not equal to $i$ or $i+1$ and so the $d_{\ell}$ face of $\Delta[n]$ contains the $i\to (i+1)$ edge, and therefore $\Lambda^{S}[n]$ contains the $i\to (i+1)$ edge of $\Delta[n]$, allowing us to apply Definition \ref{def:pinchedandaugmentedhorns}.

\begin{defn}
Given $n\geq 2$, $0\leq i\leq n-1$, and a subset $S\cin \{0,\ldots,n\}$ such that $\abs{S}\geq 2$ and exactly one of $i,i+1$ is in $S$, we say that $\Lambda^{S}[n]^{\ast}_{i\to i+1}$ is a \emph{generalized pinched horn} and that $\Lambda^{S}[n]^{\ast}_{i\to i+1}\hookrightarrow \Delta[n]^{\ast}_{i\to i+1}$ is a \emph{generalized pinched horn inclusion}. Similarly, given $X$ and an inclusion $\Delta[1]\hookrightarrow X$, we say that $\Lambda^{S}[n]^{X}_{i\to i+1}$ is a \emph{generalized $X$-augmented horn} and that $\Lambda^{S}[n]^{X}_{i\to i+1}\hookrightarrow \Delta[n]^{\ast}_{i\to i+1}$ is a \emph{generalized $X$-augmented horn inclusion}.
\end{defn}

In addition to using the 2-out-of-3 property, the proofs of the upcoming propositions rely on the fact that every Cisinski model structure is left proper by Proposition 13.1.2 in \cite{Hirschhorn}. We state this standard fact as a lemma.

\begin{lem}\label{lem:leftproper}
In a Cisinski model structure, pushouts along inclusions preserve weak equivalences.
\end{lem}

\begin{prop}\label{prop:behavediffpinched}
A Cisinski model structure is homotopically-behaved if and only if every generalized pinched horn inclusion is a weak equivalence.
\end{prop}

\begin{proof}
We first prove the forward implication. By Lemma \ref{lem:sectionsweakequiv}, we know that the composite of
\[
\begin{tikzcd}
{\Delta[n-1]} \arrow[rr, hook] &  & {\Lambda^{S}[n]^{\ast}_{i\to i+1}} \arrow[rr, hook] &  & {\Delta[n]^{\ast}_{i\to i+1}}
\end{tikzcd}
\]
is a weak equivalence, so it suffices to show that the map on the left is a weak equivalence by the 2-out-of-3 property. This map on the left is the inclusion of either the $d^i$ or $d^{i+1}$ face, whichever is in $S$.

For $n=2$, the map on the left is actually an isomorphism since $\Lambda^{S}[2]$ is the union of two 1-simplices, one of which gets collapsed to a point to create $\Lambda^{S}[2]^{\ast}_{i\to i+1}$.

For $n\geq 3$, we proceed by induction on $\abs{S}$. For the base case $\abs{S}=2$, the inclusion $\Delta[n-1]\hookrightarrow \Lambda^{S}[n]^{\ast}_{i\to i+1}$ amounts to gluing in a copy of $\Delta[n-1]^{\ast}_{i'\to i'+1}$ along one of the faces of $\Delta[n-1]$. In other words, it is a pushout of the weak equivalence $\Delta[n-2]\hookrightarrow \Delta[n-1]^{\ast}_{i'\to i'+1}$ along an inclusion, and so is a weak equivalence.

Now if $\abs{S}\geq 3$, we pick some $j\in S$ not equal to $i$ or $i+1$ and let $S'=S\smallsetminus \{j\}$. By induction we know that the inclusion $\Delta[n]\hookrightarrow \Lambda^{S'}[n+1]^{\ast}_{i\to i+1}$ is a weak equivalence, so it suffices to show that the inclusion $\Lambda^{S'}[n+1]^{\ast}_{i\to i+1}\hookrightarrow \Lambda^{S}[n+1]^{\ast}_{i\to i+1}$ is as well. But this latter inclusion is itself a pushout of a pinched generalized horn whose subset of indices is of size one less than $\abs{S}$, and so is a weak equivalence.

We now turn to the reverse implication, so let us assume that every generalized pinched horn inclusion is a weak equivalence. In particular, we can pick $S$ such that $\abs{S}=2$, and consider the composite
\[
\begin{tikzcd}
{\Delta[n-1]} \arrow[rr, hook] &  & {\Lambda^{S}[n]^{\ast}_{i\to i+1}} \arrow[rr, hook] &  & {\Delta[n]^{\ast}_{i\to i+1}}\nospace{.}
\end{tikzcd}
\]
By Lemma \ref{lem:sectionsweakequiv}, it suffices to show that every such composite map is a weak equivalence, and therefore (by the 2-out-of-3 property) to show that the map on the left is a weak equivalence. We proceed by induction on $n$. As we saw in the proof of the forward implication, in the base case $n=2$ the map on the left is an isomorphism. For $n\geq 3$, the map on the left is a pushout of $\Delta[n-2]\hookrightarrow \Delta[n-1]^{\ast}_{i\to i+1}$.
\end{proof}

We can now go one step further and show that certain generalized augmented horn inclusions must also be weak equivalences in a homotopically-behaved model structure.

\begin{prop}\label{prop:generalizedIaugmented}
Given a simplicial set $I$ and an inclusion $\Delta[1]\hookrightarrow I$, if the map $I\to \ast$ is a weak equivalence in a given Cisinski model structure, then the model structure is homotopically-behaved if and only if every generalized $I$-augmented horn inclusion is a weak equivalence in that model structure.
\end{prop}

\begin{proof}
In the diagram
\[
\begin{tikzcd}
I \arrow[d, hook] \arrow[rr, "\sim"]                               &  & \ast \arrow[d, hook]                               \\
{\Lambda^{S}[n]^{I}_{i\to i+1}} \arrow[d, hook] \arrow[rr, "\sim"] &  & {\Lambda^{S}[n]^{\ast}_{i\to i+1}} \arrow[d, hook] \\
{\Delta[n]^{I}_{i\to i+1}} \arrow[rr, "\sim"]                      &  & {\Delta[n]^{\ast}_{i\to i+1}}                     \nospace{,}
\end{tikzcd}
\]
the horizontal maps are all weak equivalences by Lemma \ref{lem:leftproper}. The bottom-left vertical map is a weak equivalence if and only if the bottom-right vertical map is by the 2-out-of-3 property, which is a weak equivalence if and only if the model structure is homotopically-behaved by Proposition \ref{prop:behavediffpinched}.
\end{proof}

Let us recall the definition of $J$, the key example of a simplicial set that is weakly equivalent to $\ast$ in every Cisinski model structure on $\sSet$.

\begin{defn}
Let $\mathbb{I}$ be the category with two objects and precisely one morphism in every hom-set, sometimes called the \emph{free-living isomorphism}. We denote the nerve of the free-living isomorphism by $J=N(\mathbb{I})$.
\end{defn}

\begin{cor}\label{cor:minimalhtpicalimpliesaughorns}
A Cisinski model structure is homotopically-behaved if and only if every generalized $J$-augmented horn inclusion is a weak equivalence.
\end{cor}

\begin{proof}
Observe that $J\to \ast$ has the right lifting property with respect to all monomorphisms, and so is weak equivalence in every Cisinski model structure. Therefore the hypothesis of Proposition \ref{prop:generalizedIaugmented} is satisfied for $\Delta[1]\hookrightarrow J$.
\end{proof}

We can further strengthen this statement once we prove the following lemma.

\begin{prop}\label{prop:generalizedhornpushoutsofhorns}
Given a simplicial set $I$ and $\Delta[1]\hookrightarrow I$, every generalized $I$-augmented horn inclusion can be realized as a sequence of pushouts of $I$-augmented horn inclusions.
\end{prop}

\begin{proof}
We emulate Joyal's proof for generalized inner horns, Proposition 2.12(iv) in \cite{Joyal:theory}. To show that every generalized $I$-augmented horn inclusion $\Lambda^{S}[n]^{X}_{i\to i+1}\hookrightarrow \Delta[n]^{\ast}_{i\to i+1}$ is a pushout of $I$-augmented horn inclusions, we proceed by induction on $k=n-\abs{S}$. The base case, $k=0$, is immediate because then $\abs{S}=n$ so $\Lambda^{S}[n]^{X}_{i\to i+1}\hookrightarrow \Delta[n]^{\ast}_{i\to i+1}$ is itself an $I$-augmented horn inclusion. For $k\geq 1$, we pick $\ell$ in $\{0,\ldots,n\}\smallsetminus (S\cup \{i, i+1\})$ and let $T=S\cup\{\ell\}$. Then we have
\[
\begin{tikzcd}
{\Lambda^{S}[n]^{X}_{i\to i+1}} \arrow[rr, hook] &  & {\Lambda^{T}[n]^{X}_{i\to i+1}} \arrow[rr, hook] &  & {\Delta[n]^{\ast}_{i\to i+1}}\nospace{.}
\end{tikzcd}
\]
The right map has $n-\abs{T}< n-\abs{S} =k$. The left map is a pushout of a generalized $I$-augmented horn inclusion with indexing set $S'$ with the same size as $S$, so that $n-1 -\abs{S'}=n-1-\abs{S} < k$. Therefore, both of these maps are sequences of pushouts of $I$-augmented horn inclusions by the inductive hypothesis, meaning the composite is as well.
\end{proof}

\begin{cor}\label{cor:characterizedhtpicalMS}
Let $\mathcal{M}$ be a Cisinski model structure on $\sSet$. Given a simplicial set $I$ and an inclusion $\Delta[1]\hookrightarrow I$, if the map $I\to \ast$ is a weak equivalence in the model structure $\mathcal{M}$, then $\mathcal{M}$ is homotopically-behaved if and only if every $I$-augmented horn inclusion is a weak equivalence in $\mathcal{M}$. In particular, the model structure $\mathcal{M}$ is homotopically-behaved if and only if every $J$-augmented horn inclusion is a weak equivalence in $\mathcal{M}$.
\end{cor}

\begin{proof}
By Proposition \ref{prop:generalizedhornpushoutsofhorns}, generalized $I$-augmented horn inclusions are weak equivalences if and only if $I$-augmented horn inclusions are. Apply this observation to Proposition \ref{prop:generalizedIaugmented} and Corollary \ref{cor:minimalhtpicalimpliesaughorns}.
\end{proof}

We conclude this subsection by comparing the lifting properties of fibrant objects in a homotopically-behaved model structure to those of quasi-categories.

\begin{defn}
Let $h\colon \sSet\to \Cat$ be the left adjoint of the nerve functor. We say an edge in a simplicial set $X$ is a \emph{categorical pre-isomorphism} if it becomes an isomorphism in the category $h(X)$.
\end{defn}

\begin{rmk}
The functor $h$ freely builds a category out of a simplicial set $X$ where the set of objects of $hX$ is the set of 0-simplices $X_0$, and the set of morphisms is generated by the 1-simplices with 2-simplices witnessing composition. We discuss $h$ in more detail in \ref{sec:special}. The key takeaway at the moment is that the universal property of the unit $X\to hX$ implies that an edge $e$ of $X$ is a categorical pre-isomorphism precisely if every map from $X$ to the nerve of a category sends $e$ to an isomorphism.
\end{rmk}

An intuitive justification for why augmented horn inclusions are weak equivalences in a homotopically-behaved model structure comes from recalling the special outer horn lifting property of quasi-categories.

\begin{prop}\label{prop:qcatsphorn}\cite[Thm.~1.3]{Joyal:published}
If $Q$ is a quasi-category, then for every $n\geq 2$ and every $u\colon \Lambda^0[n]\to Q$ such that the edge $0\to 1$ in $\Lambda^0[n]$ is sent to a categorical pre-isomorphism by $u$, we have a lift
\[
\begin{tikzcd}
{\Lambda^0[n]} \arrow[d, hook] \arrow[r, "u"] & X \\
{\Delta[n]} \arrow[ru, dotted]                &  \ \ \nospace{.}
\end{tikzcd}
\]
Similarly, for every $n\geq 2$ and every $v\colon \Lambda^n[n]\to Q$ such that the edge $n-1\to n$ in $\Lambda^n[n]$ is sent to a categorical pre-isomorphism by $v$, we have an extension of $v$ along $\Lambda^n[n]\hookrightarrow \Delta[n]$.
\end{prop}

We can interpret this result is as follows: even though quasi-categories do not necessarily have lifts of all outer horns, if we know that the $0\to 1$ edge of the horn $\Lambda^0[n]\to X$ or the $(n-1)\to n$ edge of the horn $\Lambda^n[n]\to X$ horn is ``invertible'' in $X$ in a certain sense, then we do get a lift.

This same intuition applies to $I$-augmented horn inclusions in homotopically-behaved model structures where $I\to \ast$ is a weak equivalence. Since $I$ is weakly equivalent to a point, it makes sense to think of all of the edges of $I$ as invertible. Therefore, a map $\Lambda^{j}[n]^{I}_{i\to i+1}\to X$ (where $j=i$ or $i+1$) is a horn in $X$ where we view the $i\to (i+1)$ edge as invertible. The $I$-augmented horn inclusions being weak equivalences in this model structure implies that if $X$ is fibrant, then we get a lift $\Delta[n]^{I}_{i\to i+1}\to X$ extending that horn.

\subsection{Augmented triangulations}

The goal of this subsection is to address a complication arising from the discussion above. To explain, let us first make the following definition.

\begin{defn}
Given $Z$ and an inclusion $\iota\colon\Delta[1]\hookrightarrow Z$, we say that an edge $e\colon\Delta[1]\to X$ in an arbitrary simplicial set $X$ is an \emph{$Z$-edge} if $e$ factors through $\iota$.
\end{defn}

Given a simplicial set $I$ and $\Delta[1]\hookrightarrow I$ and a homotopically-behaved model structure with $I\to\ast$ a weak equivalence, we have seen above how it makes sense to view $I$-edges in arbitrary simplicial sets as invertible. But then any good notion of ``invertible edges'' should satisfy a \emph{simplicial 2-out-of-3} property: if two edges of a 2-simplex $\Delta[2]\to X$ are invertible, then so is the third edge. The complication is that in an arbitrary simplicial set, the set of $I$-edges do not necessarily satisfy the simplicial 2-out-of-3 property (unless $I=\Delta[1]$). As a minimal counter-example, one can simply take $\Delta[2]$ itself and glue in a copy of $I$ along two of its non-degenerate edges. The takeaway is that no single $I$ can be used to identify which edges we want to view as invertible in an arbitrary simplicial set (except for the special case when $I=\Delta[1]$).

To address this concern, let us characterize the edges that we want to be invertible even if they are not $I$-edges themselves. We begin by defining \emph{unordered triangulations}.

\begin{defn}
Given $n\geq 2$ and a regular $(n+1)$-gon with vertices labeled 0 through $n$ (in no particular order), we say an \emph{unordered triangulation $\mcT$} is a decomposition of this $(n+1)$-gon into $2$-simplices such that every 0-simplex corresponds to a unique vertex of the $(n+1)$-gon and such that the 1-simplices point from lower numbers to higher numbers.
\end{defn}

Figure \ref{fig:triangulation} shows an example of an unordered triangulation of the octagon.
\begin{figure}[h]
\caption{}
\label{fig:triangulation}
\vspace{3mm}
\adjustbox{scale=1}{
\begin{tikzcd}
                        & 7 \arrow[from=ld] &  & 5 \arrow[ll] &              \\
6                       &                &  &                   & 4 \arrow[lu] \arrow[llll]\arrow[lllu] \\
                        &                &  &                   &              \\
0 \arrow[rd] \arrow[uu] &                &  &                   & 2 \arrow[uu] \\
                        & 1 \arrow[rr] \arrow[urrr]\arrow[uuurrr] \arrow[uuul]  &  & 3 \arrow[from=ru]      &             
\end{tikzcd}
}
\end{figure}

\begin{ex}\label{ex:squares}
For $n=2$, there is only one unordered triangulation of the triangle, the standard 2-simplex itself. For $n=3$, there are precisely six unordered triangulations of the square,
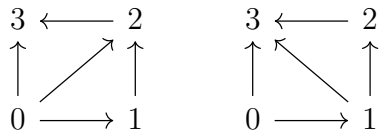
\[
\begin{tikzcd}
3                                & 2 \arrow[l] & 3                     & 2 \arrow[l]            \\
0 \arrow[r] \arrow[u] \arrow[ru] & 1 \arrow[u] & 0 \arrow[u] \arrow[r] & 1 \arrow[lu] \arrow[u]
\end{tikzcd}
\]
and
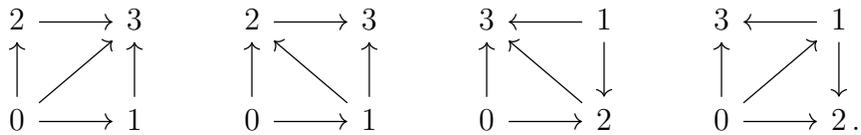
\[
\begin{tikzcd}
2 \arrow[r]                      & 3           & 2 \arrow[r]           & 3                      & 3                     & 1 \arrow[l] \arrow[d] & 3                                & 1 \arrow[d] \arrow[l] \\
0 \arrow[r] \arrow[u] \arrow[ru] & 1 \arrow[u] & 0 \arrow[u] \arrow[r] & 1 \arrow[lu] \arrow[u] & 0 \arrow[r] \arrow[u] & 2 \arrow[lu]          & 0 \arrow[r] \arrow[u] \arrow[ru] & 2\nospace{.}                   
\end{tikzcd}
\]
\end{ex}

\begin{rmk}
For those familiar with 2-Segal objects, we note that these unordered triangulations are similar to the triangulations used to define the 2-Segal condition, except that in the 2-Segal definition one requires the vertices of the $(n+1)$-gon be cyclically ordered with the exception of the $0\to n$ edge. The first two triangulations in Example \ref{ex:squares} are the triangulations of the square used in the 2-Segal condition.
\end{rmk}

We now define \emph{augmented unordered triangulations} that characterize edges that we want to view as invertible.

\begin{defn}
Given a simplicial set $I$ and $\Delta[1]\hookrightarrow I$, an \emph{$I$-augmented unordered triangulation $\mcT^I$} is an unordered triangulation $\mcT$ with a copy of $I$ glued in along all but one of the outer edges. We say that $\mcT^I$ has \emph{size $n$} if there are $n+1$ outer edges (and so $n$ copies of $I$ glued in). We consider $I$ itself to be an $I$-augmented unordered triangulation of size 1. We say that an edge $\Delta[1]\to X$ in an arbitrary simplicial set is an \emph{almost-$I$-edge} if it is a $\mcT^I$-edge for some $\mcT^I$.
\end{defn}

In the above definition, all but one of the outer edges being invertible (since they are $I$-edges) means that we should consider the remaining outer edge to be invertible as well by iterated simplicial 2-out-of-3 arguments.
\begin{figure}[h]
\caption{}
\label{fig:triangulationbreakdown}
\vspace{3mm}
\begin{tikzcd}
z &                                                                                                   & y \arrow[ll, no head] \arrow[ll, "\mathcal{U}^{I}", no head, dashed, bend right=60] \\
  & x \arrow[lu, no head] \arrow[ru, "e"'] \arrow[lu, "\mathcal{V}^{I}"', no head, dashed, bend left=74] &                                                                                   
\end{tikzcd}
\end{figure}
The idea is that if we consider $I$-edges invertible, then an edge $e\colon\Delta[1]\to X$ in an arbitrary simplicial set is forced to be invertible by iterated application of the simplicial 2-out-of-3 property precisely if it is an almost-$I$-edge. Figure \ref{fig:triangulationbreakdown} indicates the inductive argument affirming this intuition, which we spell out in the following propositions.

\begin{prop}\label{prop:almost2outof3}
Almost-$I$-edges satisfy the simplicial 2-out-of-3 property. More precisely, if $\Delta[2]\to X$ is a 2-simplex in an arbitrary simplicial set where two of the faces are almost-$I$-edges, then so is the third.
\end{prop}

\begin{proof}
Let $\mcU^I$ and $\mcV^I$ be $I$-augmented unordered triangulations that the given edges factor through. Call the remaining edge $e$. Then we can define $\mcT^I$ by gluing $\mcU^I$ and $\mcV^I$ to the appropriate faces of $\Delta[2]$ as in Figure \ref{fig:triangulationbreakdown}, making the remaining edge a $\mcT^I$-edge.
\end{proof}

\begin{prop}
If the map $I\to \ast$ is a weak equivalence in a given homotopically-behaved model structure, then so is every $\mcT^I\to \ast$ and every generalized $\mcT^I$-augmented horn inclusion.
\end{prop}

\begin{proof}
We proceed by induction on the size of $\mcT^I$. The base case of size 1, where $\mcT^I=I$, is covered by Proposition \ref{prop:generalizedIaugmented}.

For $\mcT^I$ of size bigger than 1, denote by $\begin{tikzcd} x \arrow[r, "e"] & y\end{tikzcd}$ the outer edge of $\mcT^I$ without a copy of $I$. Then we can break down $\mcT^I$ into the 2-simplex of which $e$ is a face plus $\mcU^I$ and $\mcV^I$ of smaller size, as in Figure \ref{fig:triangulationbreakdown}.
By the inductive hypothesis, we know that $\mcU^I\to \ast$ and $\mcV^I\to \ast$ are weak equivalences, so the inclusion of the point $z$ into $\mcU^I$ and $\mcV^I$ is as well, and so taking the pushout we see that $\ast \hookrightarrow \mcU^I \cup \mcV^I$ is as well. But then $\mcU^I \cup \mcV^I\hookrightarrow \mcT^I$ is a pushout of a $\mcU^I$-augmented 2-horn inclusion (or a $\mcV^I$-augmented 2-horn inclusion). Thus we see that $\ast\to \mcT^I$ is a weak equivalence and hence $\mcT^I\to \ast$ is also by the 2-out-of-3 property. By applying Proposition \ref{prop:generalizedIaugmented} to $\mcT^I$, we see that all generalized $\mcT^I$-augmented horn inclusions are weak equivalences.
\end{proof}

\begin{defn}
We say that a (generalized) $\mcT^I$-augmented horn inclusion is a \emph{(generalized) almost-$I$-augmented horn inclusion}.
\end{defn}

\begin{rmk}\label{rmk:countability}
Given a simplicial set $I$ and $\Delta[1]\hookrightarrow I$, there are countably many $I$-augmented unordered triangulations $\mcT^I$, up to isomorphism. Thus, there are countably many almost-$I$-augmented horn inclusions up to isomorphism. More generally, given any countable set of inclusions $\{\Delta[1]\hookrightarrow I_r\}_{r\geq 1}$, the set of all almost-$I_r$-augmented horn inclusions for varying $r\geq 1$ is still countable.
\end{rmk}

We have thus shown that these almost-$I$-augmented horn inclusions are forced to be weak equivalences in a homotopically-behaved model structure where $I\to \ast$ is a weak equivalence. Our next task is to show that we can apply Cisinski's machinery to this class of maps to get a model structure for certain $I$.

\section{Minimal homotopically-behaved model structures}\label{sec:behaved}

In this section we apply Cisinski's machinery to produce model structures whose fibrant objects are those with lifts of particular augmented horns. We do so using the new concept of a pointwise exact cylinder. We then show that these model structures are ``minimal'' in a certain sense, both with respect to being homotopically-behaved and with respect to the chosen exact cylinders.

\subsection{Pointwise cylinders}\label{sub:pointwise}
Our only example of an exact cylinder given above was of the form $X\mapsto I\times X$ for some simplicial set $I$. In this subsection, we describe a slightly more complex kind of exact cylinder which we use to construct our model structure.

Let $\sk_0$ denote the endofunctor of $\sSet$ that sends a simplicial set $X$ to its 0-skeleton (the simplicial set that has the same 0-simplices as $X$ but no non-degenerate higher simplices). Given a monomorphism $\iota\colon\Delta[1]\hookrightarrow I$, we let $\iota\odot X$ be the pushout
\[
\begin{tikzcd}
{\Delta[1]\times\sk_0 X} \arrow[d, "\iota\times\sk_0 X"', hook] \arrow[r, hook] & {\Delta[1]\times X} \arrow[d, hook] \\
I\times\sk_0 X \arrow[r, hook]                                             & \iota\odot X \nospace{.}                         
\end{tikzcd}
\]
In other words, the simplicial set $\iota\odot X$ is $\Delta[1]\times X$ with a copy of $I$ glued in along $\Delta[1] \times \{x\}$ for each 0-simplex $x$ of $X$.

\begin{ex}
When $X=\Delta[2]$, we glue in a copy of $I$ along the three red edges of $\Delta[1]\times \Delta[2]$ depicted in Figure \ref{fig:threerededges}.
\begin{figure}[h]
\caption{}
\label{fig:threerededges}
\vspace{3mm}
\begin{tikzcd}[row sep=small]
\bullet \arrow[rd] \arrow[rr] \arrow[dd, red] &                                  & \bullet \arrow[dd, red] \\
                                            & \bullet \arrow[ru]                    &                       \\
\bullet \arrow[rd] \arrow[rr]                    &                                  & \bullet                    \\
                                            & \bullet \arrow[ru] \arrow[from=uu, crossing over, red] &                      
\end{tikzcd}
\end{figure}
\end{ex}

\begin{prop}
The above description of $\iota\odot-$ defines an exact cylinder.
\end{prop}

\begin{proof}
Functoriality follows from $\iota\odot-$ being a pushout of the functors $I\times \sk_0(-)$, $\Delta[1]\times \sk_0(-)$, and $\Delta[1]\times -$. These three functors preserve monomorphisms and small colimits, so their pushout does as well, meaning that $\iota\odot-$ satisfies axiom (DH1) from Definition \ref{def:exactcylinder}. By Remark \ref{rmk:exactpullback}, to check axiom (DH2) we simply observe that for any inclusion of simplicial sets $A\hookrightarrow B$, the simplices of $\iota\odot B$ that are in both $\iota\odot A$ and $\{\varepsilon\}\odot B$ are precisely those in $\{\varepsilon\}\odot A$.
\end{proof}

\begin{defn}
We call the exact cylinder $\iota\odot-$ the \emph{pointwise cylinder} for the inclusion $\iota\colon \Delta[1]\hookrightarrow I$.
\end{defn}

\begin{ex}
The functor $(\id_{\Delta[1]})\odot -$ is simply the functor $\Delta[1]\times -$, since we do not glue anything extra onto the vertical edges in this case.
\end{ex}

\begin{ex}\label{ex:K}
Let $P$ be the pushout on the left below
\[\begin{tikzcd}[row sep=small, column sep=small]
	{\Delta[1]} && {\Delta[0]} && {\Delta[1]} & {\Delta[2]} & P \\
	&&&& {\Delta[2]} \\
	{\Delta[2]} && P && P && K\nospace{,}
	\arrow["{d^1}"', from=1-1, to=3-1]
	\arrow["p"', from=3-1, to=3-3]
	\arrow[from=1-3, to=3-3]
	\arrow["{s^0}", from=1-1, to=1-3]
	\arrow["p"', from=2-5, to=3-5]
	\arrow["p", from=1-6, to=1-7]
	\arrow["{d^2}"', from=1-5, to=2-5]
	\arrow["{d^0}", from=1-5, to=1-6]
	\arrow["{i_2}", from=1-7, to=3-7]
	\arrow["{i_0}"', from=3-5, to=3-7]
\end{tikzcd}\]
and then let $K$ be the pushout on the right,
with $\kappa\colon \Delta[1]\hookrightarrow K$ being the diagonal composite of the right square. We can view the simplicial set $K$ as follows
\[
\begin{tikzcd}[row sep=small]
b \arrow[dd,"f"'] \arrow[rr, dotted]        &  & b \arrow[dd,"g"] \\
                                       &  &              \\
a \arrow[rruu, "\kappa"] \arrow[rr, dotted] &  & a     \nospace{,}      
\end{tikzcd}
\]
where the dotted arrows indicate degenerate edges. We can think of $K$ as the edge $\kappa$ with a left inverse $g$ and a right inverse $f$ glued in. We will use the pointwise cylinder $\kappa\odot-$ in later sections.
\end{ex}

\begin{ex}
Recall that $J$ is the nerve of the free-living isomorphism $\mathbb{I}$. Through a slight abuse of notation, we use $J\odot -$ to denote the pointwise cylinder for the inclusion $\Delta[1]\hookrightarrow J$.
\end{ex}

\subsection{Constructing the model structures}\label{sub:augMS}

The rest of this section is devoted to proving the existence of two homotopically-behaved model structures. One is the minimal homotopically-behaved model structure, whose fibrant objects are precisely the simplicial sets with lifts of $J$-augmented horn inclusions. The other is a localization of this model structure at $K\to \ast$ (where $K$ is the simplicial set from Example \ref{ex:K}), whose fibrant objects are precisely the simplicial sets with lifts of $K$-augmented horn inclusions. The key to this result is the following proposition; the existence of our desired model structures then follows from Cisinski's Theorem \ref{thm:cisinskiMS}.

\begin{prop}\label{prop:JandKanodyne} \ 
\begin{enumerate}
    \item The almost-$J$-augmented horn inclusions together with the map $\{0\}\hookrightarrow J$ generate a $(J\odot -)$-anodyne class.
    \item The almost-$K$-augmented horn inclusions together with the maps $\{\varepsilon\}\hookrightarrow K$ for $\varepsilon=0,1$ generate a $(\kappa\odot-)$-anodyne class.
\end{enumerate}
\end{prop}

Recall that we defined a $(I\otimes -)$-anodyne class in Definition \ref{def:anodyneclass}, but in Lemma \ref{lem:altaxioms} we reformulated the axioms to be easier to check. Thus, proving this proposition amounts to verifying the axioms from Lemma \ref{lem:altaxioms}. It turns out that (An1$'$) and all but one case of (An2$'$) can be proved just as easily for arbitrary $\Delta[1]\hookrightarrow I$ in the place of $\Delta[1]\hookrightarrow J$ or $\kappa\colon\Delta[1]\hookrightarrow K$. Let us give a name to the $\Delta[1]\hookrightarrow I$ such that the remaining case of (An2$'$) is satisfied.

\begin{defn}
We say that an inclusion $\Delta[1]\hookrightarrow I$ is \emph{anodyne-ready} if the maps
\[
(\partial I \odot I)\cup (I\odot\{\varepsilon\})\hookrightarrow I\odot I
\]
for $\varepsilon=0,1$ are a sequence of pushouts of almost-$I$-augmented horn inclusions.
\end{defn}

We can thus break down the proof of Proposition \ref{prop:JandKanodyne} into the following two pieces.

\begin{prop}\label{prop:anodynereadyworks}
If $\iota\colon\Delta[1]\hookrightarrow I$ is anodyne-ready, then the almost-$I$-augmented horn inclusions together with the maps $\{\varepsilon\}\hookrightarrow I$ for $\varepsilon=0,1$ generate an $(\iota\odot-)$-anodyne class.
\end{prop}

\begin{prop}\label{prop:JandKanodyneready}
The inclusions $\Delta[1]\hookrightarrow J$ and $\kappa\colon \Delta[1]\hookrightarrow K$ are anodyne-ready.
\end{prop}

We prove Proposition \ref{prop:anodynereadyworks} in Subsection \ref{sub:anodynereadyworks} and prove Proposition \ref{prop:JandKanodyneready} in Subsection \ref{sub:JandKanodyneready}. We discuss the resulting model structures we get from Cisinski's theory in Subsection \ref{sub:resulting}.

\begin{rmk}
In our arguments below, we view the simplicial set $\Delta[1]\times \Delta[n]$ as the nerve of the poset $[1]\times [n]$
\[
\begin{tikzcd}
{(0,0)} \arrow[d] \arrow[r] & {(0,1)} \arrow[d] \arrow[r] & \ldots \arrow[r] & {(0,n-1)} \arrow[d] \arrow[r] & {(0,n)} \arrow[d] \\
{(1,0)} \arrow[r]           & {(1,1)} \arrow[r]           & \ldots \arrow[r] & {(1, n-1)} \arrow[r]          & {(1,n)} \nospace{,}
\end{tikzcd}
\]
and so we view simplices of $\Delta[1]\times \Delta[n]$ as paths in this poset. We depict the first coordinate vertically.
\end{rmk}

\subsection{Proving Proposition \ref{prop:anodynereadyworks}}\label{sub:anodynereadyworks}

Given $I$ and $\Delta[1]\hookrightarrow I$, let $A(I)$ denote the set of $I$-augmented horn inclusions plus the maps $\{\varepsilon\}\hookrightarrow I$. The inclusion $\Delta[1]\hookrightarrow I$ is anodyne-ready precisely if $A(I)$ partially satisfies axiom (An2$'$). In this subsection, we justify this terminology by showing that $A(I)$ satisfies axiom (An1$'$) and the rest of (An2$'$) for arbitrary $I$ and $\Delta[1]\hookrightarrow I$, making $\Delta[1]\hookrightarrow I$ being anodyne ready precisely the missing piece for $A(I)$ to generate an $(I\odot -)$-anodyne class.

We begin by checking that our set of maps $A(I)$ satisfies axiom (An1$'$).

\begin{lem}\label{lem:an1satisfied}
Given a simplicial set $I$ and $\Delta[1]\hookrightarrow I$, the maps $(\{\varepsilon\}\odot \Delta[n])\cup (I\odot \partial\Delta[n])\hookrightarrow I\odot \Delta[n]$ for $n\geq 1$ and $\varepsilon=0,1$ can be realized as a sequence of pushouts of $I$-augmented horn inclusions.
\end{lem}

\begin{proof}
We prove the case $\varepsilon=0$; the argument for $\varepsilon=1$ is similar. We begin by identifying which simplices of $I\odot \Delta[n]$ are neither in $\{0\}\odot \Delta[n]$ nor in $I\odot \partial\Delta[n]$. Because $n\geq 1$, the extra copies of $I$ glued in along $\Delta[1]\times\sk_0(\Delta[n])$ are already present in $I\odot\partial\Delta[n]$, and so all of the simplices not in the domain of our inclusion must be contained in $\Delta[1]\times\Delta[n]\cin I\odot\Delta[n]$.

For $0\leq j\leq n$, let $P_j$ be the $(n+1)$-simplex corresponding to the path
\[
\begin{tikzcd}[column sep=small]
{(0,0)} \arrow[r] & \ldots \arrow[r] & {(0,j-1)} \arrow[r] & {(0,j)} \arrow[d] &                     &                  &         \\
                  &                  &                     & {(1,j)} \arrow[r] & {(1,j+1)} \arrow[r] & \ldots \arrow[r] & {(1,n)}\nospace{,}
\end{tikzcd}
\]
and for $0\leq j\leq n-1$, let $Q^{j}_{j+1}$ be the $n$-simplex corresponding to the path
\[
\begin{tikzcd}[column sep=small]
{(0,0)} \arrow[r] & \ldots \arrow[r] & {(0,j-1)} \arrow[r] & {(0,j)} \arrow[rd] &                     &                  &         \\
                  &                  &                     &                    & {(1,j+1)} \arrow[r] & \ldots \arrow[r] & {(1,n)}\nospace{.}
\end{tikzcd}
\]
Let $Q_0$ denote the $n$-simplex $(1,0)\to \ldots \to (1,n)$. These simplices are precisely the simplices of $\Delta[1]\times\Delta[n]$ that are not contained in $\{0\}\times \Delta[n]$ or in $\Delta[1]\times\partial\Delta[n]$, since any other simplex either avoids both vertices $(0,j)$ and $(1,j)$ for some $0\leq j\leq n$ and so is in $\Delta[1]\times\partial\Delta[n]$, or is the $n$-simplex $(0,0)\to\ldots\to(0,n)$ that is contained in $\{0\}\times \Delta[n]$. The diagram in Figure \ref{fig:PQposet} shows how these simplices fit together, with an arrow indicating that one simplex is a face of another. It remains to describe the process by which we glue in each of these simplices via an $I$-augmented horn pushout. The red arrows indicate which pairs of simplices are attached at the same step of this process, and conversely a black arrow indicates that the simplices are glued in at different steps.
\begin{figure}[h]
\caption{}
\label{fig:PQposet}
\vspace{3mm}
\begin{tikzcd}[row sep=0em]
        & Q_0 \arrow[ld, red]                       \\
P_0     &                                              \\
        & Q^0_1 \arrow[ld, red] \arrow[lu] \\
P_1     &                                              \\
\vdots  & \vdots \arrow[ld, red] \arrow[lu]         \\
P_{n-1} &                                              \\
        & \ Q^{n-1}_n \arrow[ld, red, shorten <=-2mm] \arrow[lu, shorten <=-2mm]      \\
P_n     &                                                           
\end{tikzcd}
\end{figure}

Let us spell out this process explicitly. To attach these simplices via $I$-augmented horn pushouts, we begin with $P_n$, whose only missing face is its $d_n$ face, the $n$-simplex $Q^{n-1}_n$. This horn pushout can be realized as an $I$-augmented horn pushout because the $n\to n+1$ edge of $P_n$ is the vertical edge $(0,n)\to (1,n)$, which is an $I$-edge. We continue inductively, gluing in each $P_j$ together with its $d_j$ face for $j=n-1,n-2,\ldots,2,1,0$. The red arrows in the diagram highlight the inclusion of the $d_j$ face into each (n+1)-simplex $P_j$.
\end{proof}

We record a consequence of this lemma for later use.

\begin{cor}\label{cor:bijon0simpIaug}
Given a simplicial set $I$ and $\Delta[1]\hookrightarrow I$ and any bijective-on-0-simplices inclusion $A\hookrightarrow B$, the maps $(\{\varepsilon\}\odot B)\cup (I\odot A)\hookrightarrow I\odot B$ for $\varepsilon=0,1$ can be realized as a sequence of pushouts of $I$-augmented horn inclusions.
\end{cor}

\begin{proof}
Since $A\hookrightarrow B$ is bijective on 0-simplices, it can be witnessed as a sequence of pushouts of boundary inclusions $\partial\Delta[n]\hookrightarrow \Delta[n]$ for $n\geq 1$, and so $(\{\varepsilon\}\odot B)\cup (I\odot A)\hookrightarrow I\odot B$ can be witnessed as a sequence of pushouts of the maps $(\{\varepsilon\}\odot \Delta[n])\cup (I\odot \partial\Delta[n])\hookrightarrow I\odot \Delta[n]$ for $n\geq 1$.
\end{proof}

Having shown that $A(I)$ satisfies (An1$'$), we turn to proving that part of (An2$'$) is satisfied, which follows from the following more general lemma (by setting $I'=\mcT^I$).

\begin{lem}\label{lem:an2fortherest}
Fix $\Delta[1]\hookrightarrow I'$. For all $n\geq 2$ and $0\leq i\leq n$, if $A\hookrightarrow B$ is a pushout along
\[
\left((\{0\}\sqcup\{1\})\times\Delta[n]\right)\cup \left(\Delta[1]\times \Lambda^i[n]\right) \hookrightarrow \Delta[1]\times\Delta[n]
\]
such that either the $(\varepsilon,i-1)\to (\varepsilon,i)$ edges or the $(\varepsilon,i)\to (\varepsilon,i+1)$ edges (for $\varepsilon=0,1$) are sent to $I'$-edges in $A$, then $A\hookrightarrow B$ is a finite composite of pushouts of $I'$-augmented horn inclusions.
\end{lem}

\begin{proof}
We begin by identifying which simplices of $\Delta[1]\times\Delta[n]$ are not in $\Delta[1]\times\Lambda^i[n]$ or in $(\{0\}\sqcup\{1\})\times\Delta[n]$.

A simplex is in $\Delta[1]\times\Lambda^i[n]$ if there is some $0\leq j \leq n$ with $j\neq i$ such that the simplex avoids both of the vertices $(0,j)$ and $(1,j)$. A simplex is in $(\{0\}\sqcup\{1\})\times\Delta[n]$ if its vertices are all $0$ or all $1$ in the first coordinate.

Let us fix notation for each of the simplices that avoid satisfying both of these criteria. First, we let $P_j$ and $Q^{j}_{j+1}$ be defined as in the proof of Lemma \ref{lem:an1satisfied}. For $0\leq i,j\leq n$ such that $j\neq i$, let $R(i)_j$ be the $n$-face of $P_j$ that skips the $(\varepsilon,i)$ vertex (where $\varepsilon=0$ if $j>i$ and $\varepsilon=1$ if $j<i$). For $0\leq i\leq n$ and $0\leq j\leq n-1$ such that $j\neq i-1, i$, let $S(i)^{j}_{j+1}$ be the $(n-1)$-face of $Q^{j}_{j+1}$ that skips the $(\varepsilon,i)$ vertex (where $\varepsilon=0$ if $j>i$ and $\varepsilon=1$ if $j<i-1$). Let $S(i)^{i-1}_{i+1}$ be the $(n-1)$-simplex corresponding to the path
\[
\begin{tikzcd}[column sep=small]
{(0,0)} \arrow[r] & \ldots \arrow[r] & {(0,i-1)} \arrow[rrd] & \quad &                     &                  &         \\
                  &                  &                       &  & {(1,i+1)} \arrow[r] & \ldots \arrow[r] & {(1,n)}\nospace{.}
\end{tikzcd}
\]

Figure \ref{fig:PQRSforaughorns} shows how these simplices fit together, with an arrow indicating that one simplex is a face of another. As in Figure \ref{fig:PQposet} from the proof of Lemma \ref{lem:an1satisfied}, the different-colored arrows indicate which pairs of simplices are attached at the same step of the process below, while black arrows indicate that the simplices are glued in at different steps.

\begin{figure}[h]
\caption{}
\label{fig:PQRSforaughorns}
\vspace{3mm}
\begin{tikzcd}[row sep=small]
P_0     &                                             & R(i)_0 \arrow[ll, red]     &                                                                          \\
        & Q^0_1 \arrow[lu] \arrow[ld]                 &                               & S(i)^0_1 \arrow[lu] \arrow[ld] \arrow[ll, blue]                        \\
P_1     &                                             & R(i)_1 \arrow[ll, red]     &                                                                          \\
\vdots  & \vdots \arrow[lu] \arrow[ld]                & \vdots                        & \vdots \arrow[lu] \arrow[ld]                                             \\
P_{i-2} &                                             & R(i)_{i-2} \arrow[ll, red] &                                                                          \\
        & Q^{i-2}_{i-1} \arrow[lu] \arrow[ld]         &                               & S(i)^{i-2}_{i-1} \arrow[lu] \arrow[ld] \arrow[ll, blue]                \\
P_{i-1} &                                             & R(i)_{i-1} \arrow[ll, red] &                                                                          \\
        & Q^{i-1}_i \arrow[lu] \arrow[ld]             &                               &                                                                          \\
P_i     &                                             &                               & S(i)^{i-1}_{i+1} \arrow[luu] \arrow[ldd] \arrow[llu, blue] \arrow[lld] \\
        & Q^i_{i+1} \arrow[lu, green] \arrow[ld] &                               &                                                                          \\
P_{i+1} &                                             & R(i)_{i+1} \arrow[ll, red] &                                                                          \\
        & Q^{i+1}_{i+2} \arrow[lu] \arrow[ld]         &                               & S(i)^{i+1}_{i+2} \arrow[lu] \arrow[ld] \arrow[ll, blue]                \\
P_{i+2} &                                             & R(i)_{i+2} \arrow[ll, red] &                                                                          \\
\vdots  & \vdots \arrow[lu] \arrow[ld]                & \vdots                        & \vdots \arrow[lu] \arrow[ld]                                             \\
P_{n-1} &                                             & R(i)_{n-1} \arrow[ll, red] &                                                                          \\
        & Q^{n-1}_n \arrow[lu] \arrow[ld]             &                               & S(i)^{n-1}_n \arrow[lu] \arrow[ld] \arrow[ll, blue]                    \\
P_n     &                                             & R(i)_n \arrow[ll, red]     &                                                                         
\end{tikzcd}
\end{figure}

To describe $A\hookrightarrow B$ as a sequence of pushouts of $I'$-augmented horn inclusions, we consider the case where the $(\varepsilon,i)\to (\varepsilon,i+1)$ edges are sent to $I'$-edges in $A$. (The other case is similar.) For $0\leq j\leq i-1$ and $i+1\leq j\leq n-1$, the only face of $Q^{j}_{j+1}$ that is not already in $A$ is its $d_i$ face $S(i)^{j}_{j+1}$. When $0\leq j\leq i-1$, the $(1,i)\to (1,i+1)$ edge is in $Q^{j}_{j+1}$, and when $i+1\leq j\leq n-1$, the $(0,i)\to (0,i+1)$ edge is in $Q^{j}_{j+1}$, in both cases corresponding to the $i\to (i+1)$ edge of the $n$-simplex $Q^{j}_{j+1}$. We also have $S(i)^{i-1}_{i+1}$ as the $d_i$ face of $Q^{i-1}_i$, where the edge $(1,i)\to (1,i+1)$ is the $i\to (i+1)$ edge of $Q^{i-1}_i$. We can therefore glue in every $Q$ simplex (along with its $d_i$ face) except for $Q^i_{i+1}$ as the first steps in our sequence of $I'$-augmented horn pushouts. The pairs of simplices glued in at this step are indicated by the blue arrows in the diagram. Then for $0\leq j\leq i-1$ the $d_{i+1}$ face of $P_j$ is $R(i)_j$ and the $i+1\to i+2$ edge is $(1,i)\to (1,i+1)$, and for $i+1\leq j\leq n-1$ the $d_i$ face of $P_j$ is $R(i)_j$ and the $i\to (i+1)$ edge is $(0,i)\to (0,i+1)$, so we can now glue in every $P$ simplex (along with its missing $d_i$ or $d_{i+1}$ face) except for $P_i$ as the next steps in our sequence of pushouts. These steps are indicated by the red arrows. All that is left is $P_i$ and its $d_{i+1}$ face $Q^{i}_{i+1}$, and since the $i+1\to i+2$ edge of $P_i$ is $(1,i)\to (1,i+1)$, these remaining simplices are glued in via an $I'$-augmented horn pushout as well, indicated by the green arrow.
\end{proof}

By applying Lemma \ref{lem:an2fortherest} when $I'=\mcT^{I}$, we get the following corollary which says that $A(I)$ satisfies part of (An2$'$).

\begin{cor}\label{cor:an2fortherest}
Given a simplicial set $I$ and $\Delta[1]\hookrightarrow I$, for every almost-$I$-augmented horn inclusion
\[
\Lambda^j[n]^{\mcT^{I}}_{i\to i+1}\hookrightarrow \Delta[n]^{\mcT^I}_{i\to i+1}
\]
(where $j=i$ or $i+1$), the map
\[
(\partial J\odot \Delta[n]^{\mcT^{I}}_{i\to i+1})\cup (I\odot \Lambda^j[n]^{\mcT^{I}}_{i\to i+1})\hookrightarrow I\odot \Delta[n]^{\mcT^{I}}_{i\to i+1}
\]
is a finite composite of pushouts of almost-$I$-augmented horn inclusions.
\end{cor}

\begin{proof}
This map satisfies the hypotheses of Lemma \ref{lem:an2fortherest} when $I'=\mcT^{I}$.
\end{proof}

We have now assembled the necessary ingredients to prove Proposition \ref{prop:anodynereadyworks}, which says that $\Delta[1]\hookrightarrow I$ being anodyne-ready is indeed the missing piece for the set of maps $A(I)$ to generate an anodyne class.

\begin{proof}[Proof of Proposition \ref{prop:anodynereadyworks}]
To check that axiom (An1$'$) is satisfied, we note that the map
\[
(\{\varepsilon\}\odot \Delta[0])\cup (I\odot \partial\Delta[0])\hookrightarrow I\odot \Delta[0]
\]
for $\varepsilon=0,1$ is isomorphic to $\{\varepsilon\}\hookrightarrow I$, which takes care of the $n=0$ case. For $n\geq 1$ we apply Lemma \ref{lem:an1satisfied}.

To check that axiom (An2$'$) is satisfied, we use Corollary \ref{cor:an2fortherest} to account for all of the almost-$I$-augmented horn inclusions. The remaining maps $\{\varepsilon\}\hookrightarrow I$ are accounted for because we assumed $\Delta[1]\hookrightarrow I$ to be anodyne-ready.
\end{proof}

\subsection{Proving Proposition \ref{prop:JandKanodyneready}}\label{sub:JandKanodyneready}

We now turn to showing that $\Delta[1]\hookrightarrow J$ and $\Delta[1]\hookrightarrow K$ are anodyne-ready. An important ingredient of the proof will be the following observation about when we can upgrade ordinary horn inclusions to augmented horn inclusions.

\begin{lem}\label{lem:allalmostpushout}
Given a simplicial set $I$ and $\Delta[1]\hookrightarrow I$ and $k\geq 2$, if $\Lambda^j[k]\to X$ is a $k$-horn in some simplicial set such that all of the edges of $\Lambda^j[k]$ are sent to almost-$I$-edges, then all of the edges of $\Delta[k]$ are sent to almost-$I$-edges in $\Delta[k]\sqcup_{\Lambda^j[k]} X$. Furthermore, the inclusion $X\hookrightarrow \Delta[k]\sqcup_{\Lambda^j[k]} X$ can be witnessed as a pushout of almost-$I$-augmented horn inclusion.
\end{lem}

\begin{proof}
For $k>2$, there are no new edges added in the pushout, so the first claim is immediate. For $k=2$, since the two edges of $\Lambda^j[2]$ are sent to almost-$I$-edges, the new edge is also an almost-$I$-edge in the pushout $X\sqcup_{\Lambda^j[2]}\Delta[2]$ by applying the simplicial 2-out-of-3 property from Proposition \ref{prop:almost2outof3}. In each of these cases, the inclusion $X\hookrightarrow \Delta[k]\sqcup_{\Lambda^j[k]} X$ can be upgraded to an almost-$I$-augmented horn pushout because every edge of the horn in $X$ is an almost-$I$-edge. 
\end{proof}

Iterated application of Lemma \ref{lem:allalmostpushout} yields the following corollary.

\begin{cor}\label{cor:allalmostpushout}
Given a simplicial set $I$ and $\Delta[1]\hookrightarrow I$, if $A\hookrightarrow B$ is a sequence of pushouts of $k$-horns for varying $k\geq 2$, and $A\to X$ is a map such that all of the edges of $A$ are sent to almost-$I$-edges, then all of the edges of $B$ are sent to almost-$I$-edges in $B\sqcup_{A} X$. Furthermore, the inclusion $X\hookrightarrow B\sqcup_{A} X$ can be witnessed as a sequence of pushouts of almost-$I$-augmented horn inclusions.
\end{cor}

Another ingredient to the proof of Proposition \ref{prop:JandKanodyneready} is the following observation that the given inclusions are themselves a sequence of pushouts of ordinary $k$-horns.

\begin{lem}\label{lem:JandKhornpushouts}
The inclusions $\Delta[1]\hookrightarrow J$ and $\kappa\colon \Delta[1]\hookrightarrow K$ are obtained via a sequence of pushouts of $k$-horn inclusions for $k\geq 2$.
\end{lem}

\begin{proof}
We first consider $\Delta[1]\hookrightarrow J$. Recall that $J$ has precisely two non-degenerate $n$-simplices for all $n\geq 0$. Furthermore, for each $n$-simplex, only the $d_0$ and $d_n$ faces are non-degenerate. We may therefore inductively build $J$ from $\Delta[1]$ as follows: for the base case, we observe that $\Delta[1]\cin J$ contains both 0-simplices and one of the non-degenerate 1-simplices. Now, assuming for $n\geq 1$ that we have glued in all $(n-1)$-simplices and exactly one of the non-degenerate $n$-simplices, we may glue in one of the $(n+1)$-simplices along an $(n+1)$-horn (where $n+1\geq 2$) since it is missing exactly one of its faces. We then have all of the $n$-simplices and exactly one of the non-degenerate $(n+1)$-simplices.

Now let us consider $\kappa\colon\Delta[1]\hookrightarrow K$. We build $K$ out of $\Delta[1]$ by pushouts along outer 2-horns where the $0\to 2$ edge is degenerate.
\end{proof}

Combining Lemma \ref{lem:JandKhornpushouts} and Corollary \ref{cor:allalmostpushout} yields the following corollary.

\begin{cor}
Given a simplicial set $X$, if an edge $\Delta[1]\to X$ is an almost-$J$-edge, then the inclusion $X\hookrightarrow J\sqcup_{\Delta[1]} X$ can be witnessed as a pushout of almost-$J$-augmented horns. Similarly, if $\Delta[1]\to X$ is an almost-$K$-edge, then the inclusion $X\hookrightarrow K\sqcup_{\Delta[1]} X$ can be witnessed as a pushout of almost-$K$-augmented horns.
\end{cor}

The last ingredient to the proof of Proposition \ref{prop:JandKanodyneready} is the observation that each inclusion
\[
(\partial\Delta[1]\times I ) \cup (\Delta[1]\times \{\varepsilon\})\longhookrightarrow (\Delta[1]\times I)
\]
for $I=J,K$ and $\varepsilon=0,1$ is also obtained via a sequence of pushouts of ordinary $k$-horns. The following two lemmas handle the two cases $I=J$ and $I=K$ separately.

\begin{lem}\label{lem:claimforJ}
The inclusion $(\partial \Delta[1]\times J)\cup(\Delta[1]\times \{0\})\hookrightarrow \Delta[1]\times J$ is a sequence of pushouts of $k$-horn inclusions for varying $k\geq 2$.
\end{lem}

\begin{proof}
For each $\ell\geq 0$, let $B_{\ell}$ denote the non-degenerate $\ell$-simplex of $J$ whose initial vertex is $0$. Each $B_{\ell}$ contains all $m$-simplices for $m<\ell$, so $J=\bigcup_{\ell\geq 0} B_{\ell}$. Notice that $B_0=\{0\}$, so $(\partial \Delta[1]\times J)\cup(\Delta[1]\times \{0\})$ is precisely $(\partial \Delta[1]\times J)\cup(\Delta[1]\times B_0)$. It therefore suffices to show that each inclusion
\[
(\partial \Delta[1]\times J)\cup(\Delta[1]\times B_{\ell})\hookrightarrow (\partial \Delta[1]\times J)\cup(\Delta[1]\times B_{\ell+1})
\]
for all $\ell\geq 0$ is a sequence of pushouts of $(\ell+1)$-horns and $(\ell+2)$-horns, with $(\ell+1)$-horns only necessary when $\ell\geq 1$. Recalling the notation from Figure \ref{fig:PQRSforaughorns} in the proof of Lemma \ref{lem:an2fortherest}, we depict in Figure \ref{fig:PQRSforJodotJ} the simplices of $\Delta[1]\times B_{\ell+1}$ that are not in $\Delta[1]\times B_{\ell}$.
\begin{figure}[h]
\caption{}
\label{fig:PQRSforJodotJ}
\vspace{3mm}
\begin{tikzcd}[row sep=small]
P_0     &                                             & {}     &                                                                          \\
        & Q^0_1 \arrow[lu, red] \arrow[ld]                 &                               & {}                        \\
P_1     &                                             & R(0)_1 \arrow[ll, green]     &                                                                          \\
 & Q^1_2 \arrow[lu] \arrow[ld]                &                         & S(0)^1_2 \arrow[lu] \arrow[ld]    \arrow[ll, blue]                                         \\
P_{2} &                                             & R(0)_{2} \arrow[ll, green] &                                                                          \\
        & Q^{2}_{3} \arrow[lu] \arrow[ld]         &                               & S(0)^{2}_{3} \arrow[lu] \arrow[ld] \arrow[ll, blue]                \\
\vdots &                                             & \vdots & 
\end{tikzcd}
\end{figure}
We can proceed by gluing in the $Q^i_{i+1}$ simplices together with the $S(0)^i_{i+1}$ simplices for each $i\geq 1$ via pushouts of $(\ell+1)$-horns, indicated by the blue arrows. Note that this first step is only necessary if $\ell\geq 1$. Then we can glue in via an $(\ell+2)$-horn the $P_0$ simplex with $Q^0_1$, as indicated by the red arrow. Finally, we glue each $P_i$ together with $R(0)_i$ for each $i\geq 1$, also via $(\ell+2)$-horns, as indicated by the green arrows.
\end{proof}

\begin{lem}\label{lem:claimforK}
Given $\varepsilon=0,1$, the inclusion $(\partial \Delta[1]\times K)\cup(\Delta[1]\times \{\varepsilon\})\hookrightarrow \Delta[1]\times K$ is a sequence of pushouts of 2-horn and 3-horn inclusions.
\end{lem}

\begin{proof}
We start by including all of the missing 1-simplices, which we can do working in $\Delta[1]\times \sk_1 K$. For this proof, let us rename the $0$ and $1$ vertices of $K$ to $a$ and $b$, respectively, and let $f,g,h$ denote the non-degenerate edges of $K$. The 1-simplices of $\Delta[1]\times K$ that we start with can then be pictured as in the first diagram of Figure \ref{fig:pinkgreen}.
\begin{figure}[h]
\caption{}
\label{fig:pinkgreen}
\vspace{3mm}
\begin{tikzcd}[execute at end picture={
\foreach \Nombre in  {A1,A2,A3,A4,A5,A6,A7,A8,B1,B2,B3,B4,B5,B6,B7,B8,C1,C2,C3,C4,C5,C6,C7,C8,D1,D2,D3,D4,D5,D6,D7,D8}
  {\coordinate (\Nombre) at (\Nombre.center);}
\fill[greeo,opacity=0.3] 
  (B3) -- (B4) -- (C4) -- cycle;
\fill[pur,opacity=0.3] 
  (B5) -- (B6) -- (C6) -- cycle;
\fill[pur,opacity=0.3] 
  (B7) -- (B8) -- (C8) -- cycle;
\fill[greeo,opacity=0.3] 
  (B5) -- (C5) -- (C6) -- cycle;
\fill[pur,opacity=0.3] 
  (B7) -- (C7) -- (C8) -- cycle;
\fill[greeo,opacity=0.3] 
  (A7) -- (B7) -- (B8) -- cycle;
\fill[greeo,opacity=0.3] 
  (C7) -- (D7) -- (D8) -- cycle;
}
]
|[alias=A1]|{(0,b)} \arrow[r, "0\times f"]                                      & |[alias=B1]|{(0,a)} \arrow[r, "0\times g"] \arrow[d, "{\Delta[1]\times a}"] & |[alias=C1]|{(0,b)} \arrow[r, "0\times h"]                                      & |[alias=D1]|{(0,a)} \arrow[d, "{\Delta[1]\times a}"]         \\
|[alias=A2]|{(1,b)} \arrow[r, "1\times f"']                                      & |[alias=B2]|{(1,a)} \arrow[r, "1\times g"']                              & |[alias=C2]|{(1,b)} \arrow[r, "1\times h"']                                      & |[alias=D2]|{(1,a)}                   \\
|[alias=A3]|{} \arrow[r]                                      & |[alias=B3]|{} \arrow[r] \arrow[d, dashed] \arrow[rd, dotted] & |[alias=C3]|{} \arrow[r]                                      & |[alias=D3]|{} \arrow[d]         \\
|[alias=A4]|{} \arrow[r]                                      & |[alias=B4]|{} \arrow[r, dashed]                              & |[alias=C4]|{} \arrow[r]                                      & |[alias=D4]|{}                   \\
|[alias=A5]|{} \arrow[r] \arrow[d, dotted]                    & |[alias=B5]|{} \arrow[d] \arrow[r, dashed] \arrow[rd, dashed] & |[alias=C5]|{} \arrow[r] \arrow[d, dotted]                    & |[alias=D5]|{} \arrow[d]         \\
|[alias=A6]|{} \arrow[r]                                      & |[alias=B6]|{} \arrow[r]                                      & |[alias=C6]|{} \arrow[r]                                      & |[alias=D6]|{}                   \\
|[alias=A7]|{} \arrow[r, dashed] \arrow[d] \arrow[rd, dotted] & |[alias=B7]|{} \arrow[d, dashed] \arrow[r] \arrow[rd]         & |[alias=C7]|{} \arrow[r, dashed] \arrow[d] \arrow[rd, dotted] & |[alias=D7]|{} \arrow[d, dashed] \\
|[alias=A8]|{} \arrow[r]                                      & |[alias=B8]|{} \arrow[r]                                      & |[alias=C8]|{} \arrow[r]                                      & |[alias=D8]|{}                  
\end{tikzcd}
\end{figure}
The latter three diagrams in Figure \ref{fig:pinkgreen} show the order in which we can glue in four 2-horns to end up with all of the 1-simplices of $\Delta[1]\times K$.

From here, we must glue in the missing 3-simplices as well as the remaining 2-simplices. First, we attach the 3-simplices outlined by the edges
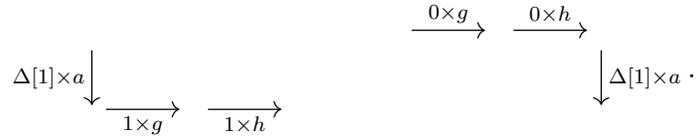
\[
\begin{tikzcd}
{} \arrow[d, "{\Delta[1]\times a}"'] &                            &    & {} \arrow[r, "0\times g"] & {} \arrow[r, "0\times h"] & {} \arrow[d, "{\Delta[1]\times a}"] \\
{} \arrow[r, "1\times g"']           & {} \arrow[r, "1\times h"'] & {} &                           &                           & {}                                 
\end{tikzcd}.
\]
The 3-simplex on the left is only missing its $d_1$ face, and the 3-simplex on the right is only missing its $d_2$ face (which are different 2-simplices), so we can attach them both via 3-horn extensions. Having done so, we can then glue in the 3-simplex
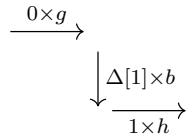
\[
\begin{tikzcd}
{} \arrow[r, "0\times g"] & {} \arrow[d, "{\Delta[1]\times b}"] &    \\
                          & {} \arrow[r, "1\times h"']          & {}
\end{tikzcd}
\]
that is now only missing its $d_0$ face (which was left empty in the diagram above). We now have all of the non-degenerate 3-simplices that involve the $h$ edges. For those involving the $f$ edges, the order
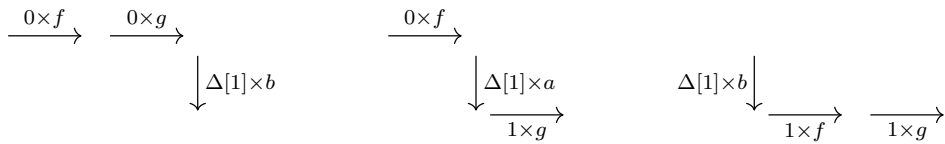
\[
\begin{tikzcd}
{} \arrow[r, "0\times f"] & {} \arrow[r, "0\times g"] & {} \arrow[d, "{\Delta[1]\times b}"] &  & {} \arrow[r, "0\times f"] & {} \arrow[d, "{\Delta[1]\times a}"] &    &  & {} \arrow[d, "{\Delta[1]\times b}"'] &                            &    \\
                          &                           & {}                                  &  &                           & {} \arrow[r, "1\times g"']          & {} &  & {} \arrow[r, "1\times f"']           & {} \arrow[r, "1\times g"'] & {}
\end{tikzcd}
\]
works because the first 3-simplex is only missing its $d_2$ face, then the second 3-simplex is only missing its $d_1$ face, and then the last 3-simplex is only missing its $d_3$ face.
\end{proof}

We now have all of the pieces to prove Proposition \ref{prop:JandKanodyneready}.

\begin{proof}[Proof of Proposition \ref{prop:JandKanodyneready}]
Let $I=J$ or $K$. We wish to show that the maps $(\partial I\odot I)\cup (I\odot \{\varepsilon\})\hookrightarrow I\odot I$ for $\varepsilon=0,1$ can be realized as a sequence of pushouts of almost-$I$-augmented horn inclusions. The central claim of the proof is that all of the missing simplices of $\Delta[1]\times I\cin I\otimes I$ can be glued in via a sequence of pushouts of $k$-horns for varying $k\geq 2$, which we proved for $I=J$ in Lemma \ref{lem:claimforJ} and for $I=K$ in Lemma \ref{lem:claimforK}. Now, we apply Lemma \ref{lem:allalmostpushout} to upgrade it to a sequence of pushouts of almost-$I$-augmented horn inclusions, and note that all of the new edges are also almost-$I$-edges. The last step is to glue in a copy of $I$ along the vertical edge that was missing when we started, i.e., the edge $\Delta[1]\times\{\varepsilon'\}$ where $\varepsilon'\neq \varepsilon$. We can witness this gluing as a sequence of pushouts of almost-$I$-augmented horn inclusions by Corollary \ref{cor:allalmostpushout}.
\end{proof}

\subsection{The resulting model structures}\label{sub:resulting}

Having proved Proposition \ref{prop:anodynereadyworks} and Proposition \ref{prop:JandKanodyneready}, and therefore Proposition \ref{prop:JandKanodyne}, Cisinski's theory gives us our desired model structures.

\begin{prop}\label{prop:JandKhtpyMS}
For $I=J$ or $K$, there is a cofibrantly generated model structure on $\sSet$ whose cofibrations are the monomorphisms, and whose fibrant objects are are the simplicial sets $X$ such that $X\to \ast$ has the right lifting property with respect to the set of almost-$I$-augmented horn inclusions.
\end{prop}

\begin{proof}
Let $S$ denote the set containing the maps $\{\varepsilon\}\hookrightarrow I$ for $\varepsilon=0,1$ together with the almost-$I$-augmented horn inclusions. By Proposition \ref{prop:JandKanodyne}, the set $S$ generates an $(I\odot -)$-anodyne class, so Theorem \ref{thm:cisinskiMS} gives us a model structure whose fibrant objects are the simplicial sets $X$ such that $S\lifts (X\to \ast)$. Since $(\ast\hookrightarrow A)\lifts(Y\to \ast)$ for all simplicial sets $A$ and $Y$, a simplicial set is fibrant in this model structure if it has lifts of almost-$I$-augmented horn inclusions.
\end{proof}

\begin{rmk}
Recall that Theorem \ref{thm:cisinskiMS} gives a description of the weak equivalences in this model structure as well, but that we do not get an explicit description of the fibrations.
\end{rmk}

To check that a simplicial set is fibrant in one of these model structures, it is actually not necessary to check all almost-$I$-augmented horn inclusions.

\begin{prop}\label{prop:almostIbecomesI}
Given $I\neq \Delta[1]$ with two 0-simplices and an inclusion $\Delta[1]\hookrightarrow I$, if $X$ is a simplicial set such that $X\to \ast$ has the right lifting property with respect to all $I$-augmented horn inclusions, then every almost-$I$-edge is an $I$-edge too.
\end{prop}

\begin{proof}
The result follows if we can show that if an edge $\Delta[1]\to X$ factors through some $\mcT^I$ then it factors through a $\mcU^I$ of strictly smaller size. Since there is at least one 2-simplex of every $I$-augmented unordered triangulation where two of the edges are $I$-edges, it suffices to show that if two edges of a 2-simplex in $X$ are $I$-edges, then so is the third edge.

Since $\Delta[1]\hookrightarrow I$ is bijective on 0-simplices, we know by Corollary \ref{cor:bijon0simpIaug} that the maps $g_{\varepsilon}\colon(\{\varepsilon\}\odot I)\cup (I\odot \Delta[1])\hookrightarrow I\odot I$ can be witnessed as sequences of pushouts of $I$-augmented horn inclusions for $\varepsilon=0,1$, so $X\to \ast$ has the right lifting property with respect to $g_{\varepsilon}$. Let $\mcW^I$ be $\Delta[2]$ with a copy of $I$ glued along two edges. Then we can choose $\varepsilon\neq \varepsilon'$ and a surjective map $f\colon(\{\varepsilon\}\odot I)\cup (I\odot \Delta[1])\to \mcW^I$ that sends the $\{\varepsilon'\}\times \Delta[1]$ edge to the edge $e$ of $\mcW^I$ that is not an $I$-edge. Let $g'\colon \mcW^I\to P$ be the pushout of $g$ along $f$, so that $X\to \ast$ also has the right lifting property with respect to $g'$. Note that the edge $e$ becomes an $I$-edge in $P$, so we have shown that every $\mcW^I$-edge of $X$ is also an $I$-edge because we can extend every $\mcW^I\to X$ along $g'\colon \mcW^I\hookrightarrow P$.
\end{proof}

\begin{cor}\label{cor:almostIaugbecomesIaug}
Let $I=J$ or $K$. A simplicial set $X$ is fibrant in the model structure from Proposition \ref{prop:JandKhtpyMS} if and only if $X\to \ast$ has the right lifting property with respect to all $I$-augmented horn inclusions.
\end{cor}

\begin{proof}
The forward implication is immediate because every $I$-augmented horn inclusion is also an almost-$I$-augmented horn inclusion. For the other implication, let us assume $X\to \ast$ has the right lifting property with respect to all $I$-augmented horn inclusions. Given an almost-$I$-augmented horn in $X$, by Proposition \ref{prop:almostIbecomesI} the almost-$I$ edge of the horn can be turned into an $I$-edge, so we get a lift.
\end{proof}

\begin{defn}
We call the model structures from Proposition \ref{prop:JandKhtpyMS} the \emph{minimal homotopically-behaved model structure} and \emph{$K$-minimal homotopically-behaved model structure}.
\end{defn}

The word ``minimal'' in the above definition is justified by the following remark.

\begin{rmk}\label{rmk:minismin}
By Corollary \ref{cor:characterizedhtpicalMS}, a Cisinski model structure is homotopically-behaved if and only if the $J$-augmented horn inclusions are weak equivalences in that model structure. The minimal homotopically-behaved model structure therefore has the smallest class of weak equivalences possible for a homotopically-behaved model structure. Similarly, if the map $K\to \ast$ is a weak equivalence in a Cisinski model structure, then that model structure is homotopically-behaved if and only if all $K$-augmented horn inclusions are weak equivalences as well, so the $K$-minimal homotopically-behaved model structure has smallest class of weak equivalences possible for a homotopically-behaved model structure where $K\to \ast$ is a weak equivalence.
\end{rmk}

\begin{cor}\label{cor:KbehavedlocalizestoJoyal}
The Joyal model structure is a localization of the $K$-minimal \break homotopically-behaved model structure.
\end{cor}

\begin{proof}
Let $X$ be a quasi-category. Given a $K$-augmented horn $\Lambda^j[n]^{K}_{i\to i+1}\to X$, there are two cases: either $0<j<n$, in which case the horn is inner so there is a lift, or the horn is an outer horn with edge $0\to 1$ or $n-1\to n$ factoring through $K$. In the latter case, the edge factoring through $K$ means it is sent to a categorical pre-isomorphism, making the horn a special outer horn, so there is a lift in this case as well.

We have shown that every fibrant object in the Joyal model structure is also fibrant in the $K$-minimal homotopically-behaved model structure, which implies that the $K$-minimal homotopically-behaved model structure is a localization of the Joyal model structure since their cofibrations are the same.
\end{proof}

Furthermore, these model structures are also minimal with respect to the exact cylinders $J\odot-$ and $\kappa\odot-$, as we now explain.

\begin{rmk}\label{rmk:retractargument}
Observe that for all $n\geq 2$, the $I$-augmented horn inclusion $\Lambda^{1}[n]^I_{0\to 1}\hookrightarrow \Delta[n]^I_{0\to 1}$ is a retract of the inclusion
\[
(\{0\}\odot \Delta[n-1])\cup (I\odot \partial\Delta[n-1])\hookrightarrow I\odot \Delta[n-1]\nospace{.}
\]
The same cannot be said of $\Lambda^{j}[n]^I_{0\to 1}\hookrightarrow \Delta[n]^I_{j-1\to j}$ for $1<j\leq n$. However, it is instead true that it is a retract of a closely related inclusion. To illustrate, recall the notation from Lemma \ref{lem:an1satisfied}, along with the diagram
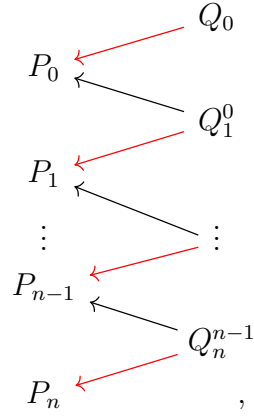
\[
\begin{tikzcd}[row sep=0em]
        & Q_0 \arrow[ld, red]                       \\
P_0     &                                              \\
        & Q^0_1 \arrow[ld, red] \arrow[lu] \\
P_1     &                                              \\
\vdots  & \vdots \arrow[ld, red] \arrow[lu]         \\
P_{n-1} &                                              \\
        & \ Q^{n-1}_n \arrow[ld, red, shorten <=-2mm] \arrow[lu, shorten <=-2mm]      \\
P_n     &                             \quad\nospace{,}                              
\end{tikzcd}
\]
showing the simplices of $\Delta[1]\odot \Delta[n]$ that are not in $(\{0\}\odot \Delta[n])\cup (I\odot \partial\Delta[n])$. Let us denote by $A_j$ the union of $(\{0\}\odot \Delta[n])\cup (I\odot \partial\Delta[n])$ with the simplices $P_{n-j}, P_{n-j+1},\ldots,P_{n}$. Then $\Lambda^{j}[n]^I_{0\to 1}\hookrightarrow \Delta[n]^I_{j-1\to j}$ is a retract of $A_j\hookrightarrow I\odot \Delta[n-1]$.
\end{rmk}

The above remark sets us up to prove the following proposition inductively.

\begin{prop}\label{prop:retractargument}
Given a simplicial set $I$ and $\Delta[1]\hookrightarrow I$ and $n\geq 2$, if the map
\[
(\{0\}\odot \Delta[n-1])\cup (I\odot \partial\Delta[n-1])\hookrightarrow I\odot \Delta[n-1]
\]
is a weak equivalence, then so is every $I$-augmented horn inclusion of the form
\[
\Lambda^{j}[n]^I_{0\to 1}\hookrightarrow \Delta[n]^I_{j-1\to j}.
\]
Similarly, if the map
\[
(\{1\}\odot \Delta[n-1])\cup (I\odot \partial\Delta[n-1])\hookrightarrow I\odot \Delta[n-1]
\]
is a weak equivalence, then so is every $I$-augmented horn inclusion of the form
\[
\Lambda^{j}[n]^I_{0\to 1}\hookrightarrow \Delta[n]^I_{j\to j+1}.
\]
\end{prop}

\begin{proof}
We prove the first claim, as the second is similar. We proceed by induction on $j$. As observed in Remark \ref{rmk:retractargument}, for the base case $j=1$, the inclusion
\[
\Lambda^{j}[n]^I_{0\to 1}\hookrightarrow \Delta[n]^I_{j-1\to j}
\]
is a retract of
\[
(\{0\}\odot \Delta[n-1])\cup (I\odot \partial\Delta[n-1])\hookrightarrow I\odot \Delta[n-1],
\]
so is a weak equivalence. Now, assuming $1<j\leq n$ and each
\[
\Lambda^{\ell}[n]^I_{0\to 1}\hookrightarrow \Delta[n]^I_{\ell-1\to \ell}
\]
for $1\leq \ell<j$ is a weak equivalence, then, using the notation from Remark \ref{rmk:retractargument}, the inclusion
\[
(\{0\}\odot \Delta[n-1])\cup (I\odot \partial\Delta[n-1])\hookrightarrow A_j
\]
is a weak equivalence, and so $A_j\hookrightarrow I\odot \Delta[n-1]$ is by the 2-out-of-3 property, and therefore
\[
\Lambda^{j}[n]^I_{0\to 1}\hookrightarrow \Delta[n]^I_{j-1\to j}
\]
is a weak equivalence since it is a retract of $A_j\hookrightarrow I\odot \Delta[n-1]$.
\end{proof}

\begin{cor}
Let $S$ be a set of monomorphisms.
\begin{enumerate}
    \item If $S$ generates an $(I\odot-)$-anodyne class for some $I$, then the corresponding Cisinski model structure $\mathcal{M}$ (whose fibrant objects are simplicial sets with the right lifting property with respect to $S$) is a localization of the minimal homotopically-behaved model structure.
    \item If $S$ generates a $(\kappa\odot-)$-anodyne class, then the corresponding Cisinski model structure $\mathcal{M}$ is a localization of the $K$-minimal homotopically-behaved model structure.
\end{enumerate}
\end{cor}

\begin{proof}
We first prove (1). The maps $(\{\varepsilon\}\odot \Delta[n-1])\cup (I\odot \partial\Delta[n-1])\hookrightarrow I\odot \Delta[n-1]$ are necessarily in ${}^{\lifts}(S^{\lifts})$ and hence weak equivalences for $\varepsilon=0,1$ and $n\geq 2$ in $\mathcal{M}$, and so by Proposition \ref{prop:retractargument} all of the $I$-augmented horn inclusions are also weak equivalences. Since the inclusion $\{0\}\hookrightarrow I$ is also in ${}^{\lifts}(S^{\lifts})$ and hence a weak equivalence, by the 2-out-of-3 property the map $I\to \ast$ is also a weak equivalence. We may therefore apply Corollary \ref{cor:characterizedhtpicalMS} to see that $\mathcal{M}$ is homotopically-behaved, and so is a localization of the minimal homotopically-behaved model structure by Remark \ref{rmk:minismin}.

To prove (2), we apply (1) with $I=K$ to see that $\mathcal{M}$ is homotopically-behaved. In our proof of (1) we also see that $K\to\ast$ is a weak equivalence in $\mathcal{M}$, so by Remark \ref{rmk:minismin} we can conclude that $\mathcal{M}$ is a localization of the $K$-minimal homotopically-behaved model structure.
\end{proof}

\begin{rmk}\label{rmk:differentMSminKbehaved}
We note that $K$-augmented horns $\Lambda^i[n]^K_{i\to i+1}$ are themselves fibrant in the minimal homotopically-behaved model structure, but not in the $K$-minimal homotopically-behaved model structure, so these model structures are distinct. Furthermore, the map $K\to \ast$ is not a weak equivalence in the minimal homotopically-behaved model structure because otherwise Proposition \ref{prop:generalizedIaugmented} would imply that all $K$-augmented horn inclusions would be as well and so the model structures would be the same.
\end{rmk}

We summarize the results of this section with the following theorem.

\begin{thm}\label{thm:summarythm}
Let $\mathcal{M}_{\operatorname{mhb}}$ be the minimal homotopically-behaved model structure on $\sSet$, and let $\mathcal{M}_{K,\operatorname{mhb}}$ be the $K$-minimal homotopically-behaved model structure on $\sSet$.
\begin{enumerate}
    \item     \begin{enumerate}
        \item Every homotopically-behaved model structure is a localization of $\mathcal{M}_{\operatorname{mhb}}$.
        \item Every Cisinski model structure corresponding to an $(I\odot -)$-anodyne class for some $I$ is a localization of $\mathcal{M}_{\operatorname{mhb}}$.
        \item The fibrant objects in $\mathcal{M}_{\operatorname{mhb}}$ are the simplicial sets with the right lifting property with respect to all $J$-augmented horn inclusions.
        \item The Joyal model structure is a localization of $\mathcal{M}_{\operatorname{mhb}}$.
    \end{enumerate}
    \vspace{3mm}
    \item      \begin{enumerate}
        \item The model structure $\mathcal{M}_{K,\operatorname{mhb}}$ is the localization of $\mathcal{M}_{\operatorname{mhb}}$ with respect to $K\to \ast$.
        \item Every Cisinski model structure corresponding to a $(K\odot -)$-anodyne class is a localization of $\mathcal{M}_{K,\operatorname{mhb}}$.
        \item The fibrant objects in $\mathcal{M}_{K,\operatorname{mhb}}$ are the simplicial sets with the right lifting property with respect to all $K$-augmented horn inclusions.
        \item The Joyal model structure is a localization of $\mathcal{M}_{K,\operatorname{mhb}}$.
    \end{enumerate}
\end{enumerate}
\end{thm}

We conclude this section with one more useful corollary to this theorem.

\begin{cor}
Given a set of monomorphisms $S=\{A_i\hookrightarrow B_i\}$ of simplicial sets such that the inclusions
\[
(A_i\times \Delta[1])\cup (B_i\times \partial\Delta[1])\hookrightarrow B_i\times \Delta[1]
\]
are in ${ }^{\lifts}(S^{\lifts})$, there exists a homotopically-behaved model structure on $\sSet$ whose fibrant objects are those with lifts against $S$ and all $J$-augmented horn inclusions.
\end{cor}

\begin{proof}
We claim that the set $S$ together with $\{0\}\hookrightarrow J$ and the set of almost-$J$-augmented horn inclusions generates a $(J\otimes -)$-anodyne class. Because we know that $\{0\}\hookrightarrow J$ together with the $J$-augmented horn inclusions generate such a class, it suffices to check that $S$ satisfies axiom (An2$'$). However, the maps that we must show are in ${ }^{\lifts}(S^{\lifts})$ are pushouts of the maps we have assumed are in ${ }^{\lifts}(S^{\lifts})$, so we are done.
\end{proof}

\section{A model structure for special horn inclusions}\label{sec:special}

Intuitively, considering $I$-edges to be ``invertible'' implies we want $I$-augmented horn inclusions to be weak equivalences. So far, the invertibility of $I$-edges has come from $I\to \ast$ being a weak equivalence. In particular, since every $\mcT^K\to \ast$ is a weak equivalence in the Joyal model structure, the almost-$K$-augmented horn inclusions are also Joyal weak equivalences. However, the Joyal model structure comes with its own notion of invertible edges, the categorical pre-isomorphisms. The following example demonstrates that, in an arbitrary simplicial set, a categorical pre-isomorphism need not be an almost-$K$-edge.

\begin{ex}\label{ex:badT}
Let $T$ be the simplicial set depicted by
\[
\begin{tikzcd}
z \arrow[d] & y \arrow[l]                                                 & y \arrow[l, dotted] \arrow[d] \\
x           & x \arrow[l, dotted] \arrow[u, "e_T" description] \arrow[lu] & w \arrow[l] \arrow[lu]       \nospace{,}
\end{tikzcd}
\]
with $e_T\colon\Delta[1]\hookrightarrow T$ the vertical edge in the middle. We have a non-degenerate 2-simplex for each of the four triangles in the picture. The dotted arrows indicate degenerate edges. If $T$ were Joyal equivalent to $\Delta[0]$, then every edge of $T$ would be a categorical pre-isomorphism, but in fact $e_T$ is the only non-degenerate categorical pre-isomorphism in $T$ as the functor $T\to \Set$
\[
\begin{tikzcd}
{\{a,c\}} \arrow[d] & \{a'\} \arrow[l, "a"']                             & \{a'\} \arrow[l, dotted] \arrow[d] \\
\{a\}               & \{a\} \arrow[l, dotted] \arrow[u] \arrow[lu, "a"'] & {\{a,c\}}\nospace{,} \arrow[l] \arrow[lu]    
\end{tikzcd}
\]
which sends every other non-degenerate edge of $T$ to a non-isomorphism in $\Set$, demonstrates.

Furthermore, there is no simplicial set $T'$ that is Joyal equivalent to $\Delta[0]$ such that the inclusion $e_T\colon \Delta[1]\hookrightarrow T$ from the above example factors through $e_{T'}\colon\Delta[1]\hookrightarrow T'$, because the sequence of edges in $T'$ that provide a left inverse to $e_{T'}$ are all categorical pre-isomorphisms, but there is no directed sequence of categorical pre-isomorphisms in $T$ from $y$ to $x$. Therefore, while all $T$-edges are necessarily categorical pre-isomorphisms, this example shows that they need not be almost-$K$-edges.
\end{ex}

The goal of this section is to address this disparity. We identify a set of inclusions $\mcE=\{\Delta[1]\hookrightarrow T\}$ such that an edge in an arbitrary simplicial set is a categorical equivalence if and only if it is a $T$-edge for some $\Delta[1]\hookrightarrow T$ in $\mcE$. We then define the set of \emph{special horn inclusions} to be the set of $T$-augmented horn inclusions for all $\Delta[1]\hookrightarrow T$ in $\mcE$. There turns out to be an intermediate model structure between the $K$-minimal homotopically-behaved model structure and the Joyal model structure where the fibrant objects are precisely the simplicial sets with lifts of special horn inclusions.

We begin by establishing notation and terminology. This first definition is standard.

\begin{defn}
For $n\geq 1$, let $Sp[n]$ denote the \emph{spine} of $\Delta[n]$, the union of the edges $i\to (i+1)$ in $\Delta[n]$ ranging over $0\leq i\leq n-1$.
\end{defn}

For the purposes of this section, we introduce some new notions in the following definitions.

\begin{defn}
For $n\geq 1$ and $1\leq i\leq n$, let $C_i[n,n+1]$ denote the union in $\Delta[n+1]$ of $Sp[n+1]$ with the 2-simplex $(i-1)\to i\to (i+1)$. Call $C_i[n,n+1]$ a \emph{composition tile}.
\end{defn}

\begin{rmk}
There are two inclusions of spines into the composition tile $C_i[n,n+1]$ that preserve the initial and final vertex, the inclusion $Sp[n+1]\hookrightarrow C_i[n,n+1]$ that hits every vertex of $C_i[n,n+1]$ and the inclusion $Sp[n]\hookrightarrow C_i[n,n+1]$ that avoids the $i$th vertex. These inclusions are depicted in Figure \ref{fig:comptile}.
\begin{figure}[h]
\caption{}
\label{fig:comptile}
\vspace{3mm}
\begin{tikzcd}[row sep=small, column sep=small]
{Sp[n]}                                    &  & {} \arrow[r, "{C_i[n,n+1]}", phantom]     & {}           &  & {Sp[n+1]}                                  \\
n                                          &  & n+1 \arrow[ll, no head, dotted]           &              &  & n+1 \arrow[lll, no head, dotted]           \\
\vdots \arrow[u]                           &  & \vdots \arrow[u]                          &              &  & \vdots \arrow[u]                           \\
i \arrow[u]                                &  & i+1 \arrow[u] \arrow[ll, no head, dotted] &              &  & i+1 \arrow[u] \arrow[lll, no head, dotted] \\
                                           &  &                                           & i \arrow[lu] &  & i \arrow[u] \arrow[ll, no head, dotted]    \\
i-1 \arrow[uu] \arrow[rr, no head, dotted] &  & i-1 \arrow[ru] \arrow[uu]                 &              &  & i-1 \arrow[u] \arrow[lll, no head, dotted] \\
\vdots \arrow[u]                           &  & \vdots \arrow[u]                          &              &  & \vdots \arrow[u]                           \\
0 \arrow[u] \arrow[rr, no head, dotted]    &  & 0 \arrow[u] \arrow[rrr, no head, dotted]  &              &  & 0 \arrow[u]                               
\end{tikzcd}
\end{figure}
\end{rmk}

\begin{defn}
Call a simplicial set a \emph{composition tiling} if it is a colimit of a diagram of the form
\[
\adjustbox{scale=0.7}{
\begin{tikzcd}
{C_i[n,n+1]} &                                            & {C_{i'}[n',n'+1]} &                                             & \ldots &                                                    & {C_{i^{(k)}}[n^{(k)},n^{(k)}+1]} \\
             & {Sp[m]} \arrow[lu, hook'] \arrow[ru, hook] &                   & {Sp[m']} \arrow[lu, hook'] \arrow[ru, hook] &        & {Sp[m^{(k-1)}]} \arrow[lu, hook'] \arrow[ru, hook] &                                 
\end{tikzcd}
}
\]
built out of the inclusions from above. (For such a diagram to make sense, we must have $n^{(j)}=n^{(j+1)}\pm 1$ and $m^{(j)}=\max (n^{(j)},n^{(j+1)})$ for all $0\leq j\leq k-1$.)

A composition tiling $C$ comes with two important inclusions of spines, coming from the unused inclusions of the composition tiles on the left and right in the diagram above. These spines must start at the same vertex and end at the same vertex, and their union is precisely the outer edges of the composition tiling. We view a composition tiling as linking these two spines. For our purposes, those two spines are the crucial data to keep track of in a composition tiling, so we use $C^{r,s}$ to denote a composition tiling linking a length $r$ spine to a length $s$ spine.
\end{defn}

\begin{ex}\label{ex:compositiontiling}
We visualize the components of the diagram
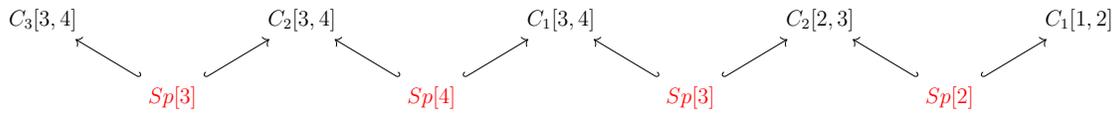
\[
\adjustbox{scale=0.7}{
\begin{tikzcd}
{C_3[3,4]} &                                            & {C_2[3,4]} &                                            & {C_1[3,4]} &                                            & {C_2[2,3]} &                                            & {C_1[1,2]} \\
           & {\color{red}{Sp[3]}} \arrow[lu, hook'] \arrow[ru, hook] &            & {\color{red}{Sp[4]}} \arrow[lu, hook'] \arrow[ru, hook] &            & {\color{red}{Sp[3]}} \arrow[ru, hook] \arrow[lu, hook'] &            & {\color{red}{Sp[2]}} \arrow[ru, hook] \arrow[lu, hook'] &           
\end{tikzcd}
}
\]
along with the leftmost and rightmost spines as shown in Figure \ref{fig:ungluedcomptiling},
\begin{figure}[h]
\caption{}
\label{fig:ungluedcomptiling}
\vspace{3mm}
\begin{tikzcd}[column sep=small]
{} \arrow[rr, no head, dotted]                    &               & {} \arrow[r, no head, dotted]                       & {} \arrow[r, no head, dotted]                    & {} \arrow[rr, no head, dotted]                       &                                          & {} \arrow[rr, no head, dotted]                    &               & {} \arrow[r, no head, dotted]                       & {} \arrow[rr, no head, dotted]                    &                & {} \arrow[r, no head, dotted]                        & {} \arrow[rr, no head, dotted]                    &                 & {} \arrow[r, no head, dotted]                           & {}                       \\
{} \arrow[u, red] \arrow[r, no head, dotted]   & {} \arrow[ru] &                                                     &                                                  &                                                      &                                          &                                                   &               &                                                     &                                                   &                &                                                      &                                                   &                 &                                                         &                          \\
{} \arrow[u, red] \arrow[rr, no head, dotted]  &               & {} \arrow[lu] \arrow[uu] \arrow[r, no head, dotted] & {} \arrow[uu, red] \arrow[r, no head, dotted] & {} \arrow[uu] \arrow[rr, no head, dotted]            &                                          & {} \arrow[uu, red] \arrow[rr, no head, dotted] &               & {} \arrow[r, no head, dotted] \arrow[uu]            & {} \arrow[uu, red] \arrow[r, no head, dotted]  & {} \arrow[ruu] &                                                      &                                                   &                 &                                                         &                          \\
                                                  &               &                                                     &                                                  &                                                      & {} \arrow[lu] \arrow[r, no head, dotted] & {} \arrow[u, red] \arrow[rr, no head, dotted]  &               & {} \arrow[r, no head, dotted] \arrow[u]             & {} \arrow[u, red] \arrow[rr, no head, dotted]  &                & {} \arrow[lu] \arrow[uuu] \arrow[r, no head, dotted] & {} \arrow[uuu, red] \arrow[r, no head, dotted] & {} \arrow[ruuu] &                                                         &                          \\
{} \arrow[uu, red] \arrow[rr, no head, dotted] &               & {} \arrow[r, no head, dotted] \arrow[uu]            & {} \arrow[uu, red] \arrow[r, no head, dotted] & {} \arrow[uu] \arrow[ru] \arrow[rr, no head, dotted] &                                          & {} \arrow[u, red] \arrow[r, no head, dotted]   & {} \arrow[ru] &                                                     &                                                   &                &                                                      &                                                   &                 &                                                         &                          \\
{} \arrow[u, red] \arrow[rr, no head, dotted]  &               & {} \arrow[r, no head, dotted] \arrow[u]             & {} \arrow[u, red] \arrow[r, no head, dotted]  & {} \arrow[u] \arrow[rr, no head, dotted]             &                                          & {} \arrow[u, red] \arrow[rr, no head, dotted]  &               & {} \arrow[lu] \arrow[uu] \arrow[r, no head, dotted] & {} \arrow[uu, red] \arrow[rr, no head, dotted] &                & {} \arrow[r, no head, dotted] \arrow[uu]             & {} \arrow[uu, red] \arrow[rr, no head, dotted] &                 & {} \arrow[luu] \arrow[uuuuu] \arrow[r, no head, dotted] & {} \arrow[uuuuu, red]\nospace{,}
\end{tikzcd}
\end{figure}
and then taking the colimit we get a composition tiling $C^{4,1}$
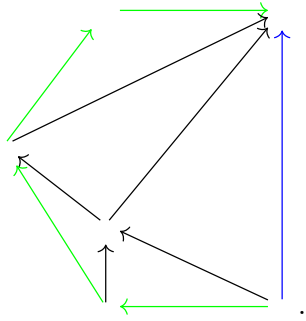
\[
\begin{tikzcd}
                                     & {} \arrow[rr, green]            &  & {}                                                     \\
                                     &                                  &  &                                                        \\
{} \arrow[ruu, green] \arrow[rrruu] &                                  &  &                                                        \\
                                     & {} \arrow[lu] \arrow[rruuu]      &  &                                                        \\
                                     & {} \arrow[luu, green] \arrow[u] &  & {} \arrow[ll, green] \arrow[llu] \arrow[uuuu, blue]\nospace{\ .}
\end{tikzcd}
\]
The green edges show the spine $Sp[4]\hookrightarrow C^{4,1}$ and the blue edge shows $Sp[1]\hookrightarrow C^{4,1}$.
\end{ex}

\begin{ex}
An unordered triangulation need not be a composition tiling. For example, in the unordered triangulation
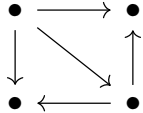
\[
\begin{tikzcd}
	\bullet & \bullet \\
	\bullet & \bullet
	\arrow[from=1-1, to=2-2]
	\arrow[from=1-1, to=2-1]
	\arrow[from=2-2, to=2-1]
	\arrow[from=2-2, to=1-2]
	\arrow[from=1-1, to=1-2]
\end{tikzcd}
\]
there is not a choice of precisely two spines whose union is the set of outer edges. However, every (ordered) triangulation of the $(n+1)$-gon is a composition tiling, linking the spine $0\to 1\to \ldots (n-1)\to n$ with the spine $0\to n$. At the same time, not every composition tiling from $Sp[n]$ to $Sp[1]$ is a triangulation since in general composition tilings can have interior vertices.
\end{ex}

\begin{rmk}\label{rmk:compositiontilingswork}
Recall that $h\colon \sSet\to\Cat$ is the left adjoint of the nerve functor. Given a simplicial set $X$, we can construct $hX$ explicitly by first letting the set of objects of $hX$ equal the set of 0-simplices $X_0$. To define $\Hom_{hX}(x,y)$ for $x,y\in X_0$, we take the set of all maps $Sp[n]\to X$ (for varying $n$) that start at $x$ and end at $y$ and then quotient out by the equivalence relation where $f\colon Sp[r]\to X$ is equivalent to $g\colon Sp[s]\to X$ if there exists a composition tiling $C^{r,s}$ and a map $C^{r,s}\to X$ such that restricting along $Sp[r]\hookrightarrow C^{r,s}$ is $f$ and restricting along $Sp[s]\hookrightarrow C^{r,s}$ is $g$. The composition functions are induced by concatenation of spines.

Although this construction of $hX$ is nonstandard, one can check that is just another way of phrasing the more standard explicit construction given in \cite{Rezk}.
\end{rmk}

\begin{defn}
For any $r\geq 1$ and composition tiling $C^{r,1}$, call the pushout
\[
\begin{tikzcd}
{Sp[1]=\Delta[1]} \arrow[d, hook] \arrow[r] & {\Delta[0]} \arrow[d, hook] \\
{C^{r,1}} \arrow[r]                         & \widetilde{C}^r            
\end{tikzcd}
\]
a \emph{pinched tiling}. Call the inclusions $\Delta[1]\hookrightarrow Sp[r]\hookrightarrow \widetilde{C}^r$ coming from the $0\to 1$ and $(r-1)\to r$ edge inclusions $\Delta[1]\to Sp[r]$ the \emph{first edge inclusion} and \emph{last edge inclusion}, respectively.
\end{defn}

\begin{ex}
In Example \ref{ex:compositiontiling}, we collapse the rightmost arrow of $C^{4,1}$ to a degeneracy to get a pinched tiling $\widetilde{C}^4$. The first edge inclusion is the bottom-most edge, and the last edge inclusion is the top-most edge in the picture.
\end{ex}

\begin{ex}
The standard 2-simplex $\Delta[2]$ is itself a composition tiling $C^{2,1}$. We collapse the $0\to 2$ edge to get a pinched tiling $\widetilde{C}^2$, whose first edge inclusion is $0\to 1$ and last edge inclusion is $1\to 2$.
\end{ex}

Since a map from a composition tiling $C^{r,1}\to X$ is capturing that the restriction to $Sp[r]\to X$ and to $\Delta[1]\to X$ correspond to the same morphism in $hX$, we can see that a map from the respective pinched tiling $\widetilde{C}^r\to X$ is capturing that the restriction to $Sp[r]\to X$ becomes the identity in $hX$. In particular, the first edge inclusion of a pinched tiling (the $0\to 1$ edge of the spine $Sp[r]$) has a left inverse (coming from the $1\to 2\to \ldots \to r$ edges), and the last edge inclusion (the $(r-1)\to r$ edge) has a right inverse ($0\to 1\to \ldots \to (r-1)$). In fact, by Remark \ref{rmk:compositiontilingswork}, an edge $\Delta[1]\to X$ has a right or left inverse in $hX$ if and only if it extends along a pinched tiling. We record this observation as a lemma.

\begin{lem}\label{lem:leftrightpreinv}
Given an edge $e\colon \Delta[1]\to X$, the morphism $h(e)$ has left (right) inverse in $hX$ if and only if there exists a pinched tiling $\widetilde{C}^r$ and a map $\widetilde{C}^r\to X$ that restricts to $e$ along the first (last) edge inclusion.
\end{lem}

Since we can use pinched tilings to identify edges of a simplicial set $X$ that have left or right inverses, we can use a pushout of pinched tilings to identify edges that have both inverses.

\begin{defn}
Given two pinched tilings $\widetilde{C}^{r}$ and $(\widetilde{C}')^{s}$, let $T$ be the pushout
\[
\begin{tikzcd}
{\Delta[1]} \arrow[d, "\text{last}"', hook] \arrow[r, "\text{first}", hook] & (\widetilde{C}')^s \arrow[d, hook] \\
\widetilde{C}^r \arrow[r, hook]                                             & T\nospace{.}                                 
\end{tikzcd}
\]
Call $T$ an \emph{inverting tiling}, and let $e_T$ denote the diagonal composite map $\Delta[1]\hookrightarrow T$. Call $e_T$ the \emph{inverting inclusion} of $T$.
\end{defn}

The following proposition shows that categorical pre-isomorphisms are characterized by maps out of inverting tilings.

\begin{prop}
An edge $e\colon \Delta[1]\to X$ is a categorical pre-isomorphism if and only if there exists an inverting tiling $T$ such that $e$ extends along the inclusion $e_T$.
\[
\begin{tikzcd}
{\Delta[1]} \arrow[d, "e_T"'] \arrow[r, "e"] & X \\
T \arrow[ru, dotted]                         &  
\end{tikzcd}
\]
\end{prop}

\begin{proof}
By Lemma \ref{lem:leftrightpreinv}, the morphism $h(e)$ has a left and a right inverse if and only if there exist two pinched tilings such that $e$ extends along the first inclusion of one and the last inclusion of the other, which happens if and only if $e$ extends along the pushout of those inclusions.
\end{proof}

\begin{ex}
The simplicial set $K$ from Example \ref{ex:K} is an inverting tiling built out of the pinched tiling $\widetilde{C}^2$.
\end{ex}

This set of inverting tilings $\{T\}$ characterizes which edges of a simplicial set we want to think of as invertible, in the context of the Joyal model structure. So, in the spirit of Section \ref{sec:behaved}, let us consider the set of $T$-augmented horn inclusions from Definition \ref{def:pinchedandaugmentedhorns}.

\begin{defn}
Let $\mcE=\{\Delta[1]\hookrightarrow T\}$ be the set of all inverting inclusions into inverting tilings. Given $\Delta[1]\hookrightarrow T$ in $\mcE$, we say that a $T$-augmented horn inclusion is a \emph{special horn inclusion}. If the horn is outer, we say that it is a \emph{special outer horn inclusion}. Let $\SpHorn$ be the set of all special horn inclusions and let $\SpOutHorn$ be the set of all special outer horn inclusions.
\end{defn}

\begin{rmk}\label{rmk:sphorncountable}
The sets $\SpHorn$ and $\SpOutHorn$ are countable by Remark \ref{rmk:countability}.
\end{rmk}

Recall that the standard phrasing of the special outer horn lifting property of quasi-categories is that there exist lifts of outer horns so long as they satisfy the additional property that a certain edge is sent to a categorical pre-isomorphism. We can now rephrase this condition directly as a lifting condition with respect to the set of special horn inclusions.

\begin{prop}\label{prop:qcatsphorn2}
If $Q$ is a quasi-category, then $Q\to \ast$ has the right lifting property with respect to $\SpOutHorn$.
\end{prop}

\begin{proof}
All inner special horns are pushouts of ordinary inner horns, so it suffices just to check outer special horns. By symmetry, it suffices to consider an inverting inclusion $\Delta[1]\hookrightarrow T$ and a special horn map $\Lambda^0[n]^T_{0\to 1}\to Q$. Because the edge $0\to 1$ factors through $T\to Q$, it is sent to a categorical pre-isomorphism in $Q$, and so by the special horn lifting property, we get an extension of the horn $\Lambda^0[n]\to Q$ to $\Delta[n]\to Q$, inducing a lift of the original special horn map $\Delta[n]^T_{0\to 1}\to Q$.
\end{proof}

We now turn to constructing the special horn model structure using Cisinski's theory.

\begin{lem}\label{lem:sphornAn2'}
The class generated by the set of special horn inclusions satisfies axiom (An2$'$) from Lemma \ref{lem:altaxioms}.
\end{lem}

\begin{proof}
Apply Lemma \ref{lem:an2fortherest} where $I'=T$, the inverting tiling for a given special horn.
\end{proof}

\begin{cor}\label{cor:sphornanodyne}
The class generated by $\SpHorn\cup \{\{\varepsilon\}\hookrightarrow K\}$ together with the set of almost-$K$-augmented horn inclusions is $(\kappa\odot-)$-anodyne.
\end{cor}

\begin{proof}
We already knew from Proposition \ref{prop:JandKanodyne} that axiom (An1$'$) is satisfied, as well as (An2$'$) for $\{\{\varepsilon\}\hookrightarrow K\}$ together with the set of almost-$K$-augmented horn inclusions. Lemma \ref{lem:sphornAn2'} tells us that (An2$'$) is satisfied for the remaining maps.
\end{proof}

\begin{thm}\label{thm:sphornMS}
There is a Cisinski model structure on $\sSet$ whose fibrant objects are are the simplicial sets $X$ such that $X\to \ast$ has the right lifting property with respect to the set of special horn inclusions.
\end{thm}

\begin{proof}
By Theorem \ref{thm:cisinskiMS}, we get a Cisinski model structure from Corollary \ref{cor:sphornanodyne} whose fibrant objects are those with lifts against $\SpHorn\cup \{\{\varepsilon\}\hookrightarrow K\}$ as well as the set of $K$-augmented horn inclusions. We claim that simply knowing $X\to \ast$ has lifts of special horn inclusions is enough to conclude that $X$ is fibrant in this model structure. We first note that $X\to \ast$ has the right lifting property with respect to $\{\{\varepsilon\}\hookrightarrow K\}$ for all simplicial sets $X$. Now, note that $K$ is itself an inverting tiling, so if $X\to\ast$ has lifts of special horn inclusions, it in particular has lifts of $K$-augmented horn inclusions, so by Corollary \ref{cor:almostIaugbecomesIaug} we see that $X\to \ast$ has lifts of all almost-$K$-augmented horn inclusions.
\end{proof}

\begin{defn}
We call the model structure in Theorem \ref{thm:sphornMS} the \emph{special horn model structure}. We say a simplicial set is \emph{special horn fibrant} if it is fibrant in this model structure.
\end{defn}

\begin{rmk}\label{rmk:specialdifferentMS}
The special outer horn lifting property of quasi-categories implies that the Joyal model structure is a localization of the special horn model structure. While these model structures have a close relationship in sharing a notion of ``invertible edges,'' they are distinct because $\Lambda^1[2]$ is fibrant in the special horn model structure but not in the Joyal model structure.

The fact that $K$ is itself an inverting tiling means that every special horn fibrant simplicial set has lifts of $K$-augmented horns. Therefore, the special horn model structure is a localization of the $K$-minimal homotopically-behaved model structure. These model structures are also distinct; if $T$ is as in Example \ref{ex:badT}, then the special horns $\Lambda^i[n]^T_{i\to i+1}$ are fibrant in the $K$-minimal homotopically-behaved model structure but not in the special horn model structure.

The special horn model structure is therefore a curious intermediate between the $K$-minimal homotopically-behaved model structure and the Joyal model structure. The fibrant objects are very similar to those of the $K$-minimal homotopically-behaved model structure, making it tempting to claim it as a model structure with ``the homotopical properties of quasi-categories without the composition aspects.'' However, compositionality actually does play a subtle but key role in determining the notion of homotopy for the special horn model structure.
\end{rmk}

We conclude this section by conjecturing a partial characterization of the trivial cofibrations in the Joyal model structure.

\begin{conj}\label{conj:joyalbijon0simp}
The class of trivial cofibrations in the Joyal model structure that are bijective-on-0-simplices is generated by the set of inner horn inclusions together with the set of special outer horn inclusions. That is, the bijective-on-0-simplices trivial cofibrations are precisely $^{\lifts}((\InnHorn\cup \SpOutHorn)^{\lifts})$.
\end{conj}

\begin{rmk}
The special outer horn inclusions are weak equivalences in the special horn model structure, and so are also Joyal weak equivalences. The uncertain aspect of the conjecture is whether the containment of $^{\lifts}((\InnHorn\cup \SpOutHorn)^{\lifts})$ in the class of bijective-on-0-simplices trivial cofibrations in the Joyal model structure is strict.
\end{rmk}

Joyal \cite{Joyal:notes} left open whether the inner horn inclusions alone generated the bijective-on-0-simplices trivial cofibrations in his model structure, but Campbell \cite{Campbell} recently provided a counter-example of a map that is bijective-on-0-simplices and a weak equivalence in the Joyal model structure but is not in $^{\lifts}(\InnHorn^{\lifts})$. Since Campbell's map is in fact a pushout of a special 2-horn, it is not a counter-example to Conjecture \ref{conj:joyalbijon0simp}.

An intuitive argument for Conjecture \ref{conj:joyalbijon0simp} is that this set of maps seems as close as possible to the set of ordinary horn inclusions (which generate the trivial cofibrations of the Kan-Quillen model structure on $\sSet$) while still being weak equivalences in the Joyal model structure.

We also state a similar conjecture for the special horn model structure.

\begin{conj}\label{conj:sphornbijon0simp}
The class of bijective-on-0-simplices trivial cofibrations in the special horn model structure is generated by the set of special horn inclusions. That is, the bijective-on-0-simplices trivial cofibrations are precisely $^{\lifts}(\SpHorn^{\lifts})$.
\end{conj}

\section{Comparing model structures}\label{sec:comparison}

In this section, we compare the fibrant objects in the minimal model structure to the fibrant objects of homotopically-behaved model structures to get a better understanding of what it means to be homotopically-behaved. We begin by explaining the horn-based characterization of the minimal model structure's fibrant objects from \cite{Feller:minimal}.

\begin{defn}\label{def:isoplex}
Fix $n\geq 1$ and $0\leq i\leq n-1$. Let $[n]_{i}$ denote the category $[n]$ with the morphism $c_{i}\to c_{i+1}$ inverted
\[
\begin{tikzcd}[column sep=small]
c_0 \arrow[r] & c_1 \arrow[r] & \ldots \arrow[r] & c_{i-1} \arrow[r] & c_i \arrow[r, shift left] & c_{i+1} \arrow[l, shift left] \arrow[r] & c_{i+2} \arrow[r] & \ldots \arrow[r] & c_{n-1} \arrow[r] & c_n\nospace{,}
\end{tikzcd}
\]
and let $\tri_{i}[n]$ denote the nerve of $[n]_i$. Call $\tri_{i}[n]$ an \emph{$n$-isoplex}, or simply an \emph{isoplex}.

Let $d_j \tri_{i}[n]$ denote the nerve of the full subcategory of $[n]_{i}$ that includes all but the $j$th vertex. Call $d_j \tri_{i}[n]$ the \emph{$j$th face} of $\tri_{i}[n]$.
\end{defn}

For $j\neq i, i+1$, the $d_j$ face of the isoplex $\tri_{i}[n]$ is an $(n-1)$-isoplex, while the $d_i$ and $d_{i+1}$ faces are standard $(n-1)$-simplices. We can therefore think of the $n$-isoplex as an ``isomorphism of $(n-1)$-simplices'' between its $d_{i+1}$ face and its $d_i$ face.

Having defined our ``iso'' analogue of simplices and faces, we can now define ``iso-horns'' to be the union of all but one face of an isoplex. However, just like we saw with augmented horns, we want to limit ourselves to horns that omit the $d_k$ or $d_{k+1}$ face, where $k\to k+1$ is a $J$-edge. Furthermore, due to the symmetry of isoplexes, it suffices just to consider horns where the $d_k$ face is missing.

\begin{defn}\label{def:isohorn}
Let $n\geq 1$ and $0\leq i\leq n-1$ be as in Definition \ref{def:isoplex}. Let $\bV_{i}[n]$ be the union of all but the $i$th face of $\tri_{i}[n]$. We call $\bV_{i}[n]$ an \emph{iso-horn}, and call the inclusion $\bV_{i}[n]\hookrightarrow \tri_{i}[n]$ an \emph{iso-horn inclusion}.
\end{defn}

We can now state the main result of \cite{Feller:minimal}.

\begin{thm}\label{thm:isohornminimalfibrant}
A simplicial set $X$ is fibrant in the minimal model structure if and only if it has lifts of iso-horn inclusions.
\end{thm}

The takeaway of this theorem is that, from a certain perspective, isoplexes are the fundamental building blocks of homotopies in the minimal model structure, and hence in any Cisinski model structure. They are inherently equipped with all the higher invertibility data we want from a good notion of homotopy.

To understand what is happening when we localize to a homotopically-behaved model structure, consider the diagram
\[
\begin{tikzcd}
	{\Delta[n]^J_{i\to i+1}} & {\tri_i[n]} \\
	& {\Delta[n-1]}\nospace{.}
	\arrow[hook, from=1-1, to=1-2]
	\arrow[from=1-1, to=2-2]
	\arrow["\sim", from=1-2, to=2-2]
\end{tikzcd}
\]
Since the map on the right is a weak equivalence in any Cisinski model structure, by localizing with respect to the maps on the left to yield a homotopically-behaved model structure, we are equivalently localizing with respect to the horizontal maps. These maps being weak equivalences is effectively saying that ``$n$-simplices with a $J$-edge along $i\to (i+1)$ extend to full-fledged homotopies of $(n-1)$-simplices.'' In other words, we are justified in considering a $n$-simplex with a $J$-edge in its spine to be a ``homotopy'' since all of the higher invertibility data comes along for free.

We conclude by reviewing the broader picture of Cisinski model structures on $\sSet$ that localize to the Joyal model structure. Figure \ref{fig:MSno2seg} shows the model structures discussed in this paper, with an arrow drawn to indicate that the target is a localization of the source. (The minimal model structure is denoted by $\Min(\varnothing)$, and its localization with respect to $K\to \ast$ by $\Min(K\to \ast)$. The minimal and $K$-minimal homotopically-behaved model structures are depicted in the second column.)
\begin{figure}[h]
\caption{}
\label{fig:MSno2seg}
\vspace{3mm}
\begin{tikzcd}
\Min(\varnothing) \arrow[d] \arrow[r] & \Min(\text{HtpyBehaved}) \arrow[d] &                   &       \\
\Min(K\to \ast) \arrow[r]             & \Min(K,\text{HtpyBehaved}) \arrow[r]             & \text{Special} \arrow[r] & \text{Joyal}
\end{tikzcd}
\end{figure}
By Remark \ref{rmk:specialdifferentMS}, the two localizations on the right are nontrivial. The vertical arrows indicate nontrivial localizations because $K\to \ast$ is not a weak equivalence in the two upper model structures by Remark \ref{rmk:differentMSminKbehaved}. The upper horizontal arrow indicates a nontrivial localization because the simplicial set $\Delta[2]^{\ast}_{0\to 1}$ is fibrant in the minimal model structure but not in the minimal homotopically-behaved model structure. The remaining localization is nontrivial by the following lemma.

\begin{lem}
Given a simplicial set $I$ with exactly two vertices $a$ and $b$ and with 1-simplices $a\to b$ and $b\to a$, the localization of the minimal model structure at the map $I\to \ast$ is not homotopically-behaved.
\end{lem}

\begin{proof}
Since the map $\Delta[2]^{\ast}_{0\to 1}\to \ast$ is a weak equivalence but not a trivial fibration in a homotopically-behaved model structure, it cannot be a fibration. So, it suffices to show that $\Delta[2]^{\ast}_{0\to 1}$ is fibrant in the minimal model structure localized at $I\to \ast$. A similar argument as in Section 3 of \cite{Feller:minimal} shows that a simplicial set is fibrant in this model structure if and only if it has lifts with respect to all maps in the set
\[
\mcA_I=\{(I\times \partial \Delta[n])\cup(\{v\}\times \Delta[n])\hookrightarrow I\times \Delta[n]\}_{v\in\{a,b\}, n\geq 0}.
\]
We therefore see that $\Delta[2]^{\ast}_{0\to 1}$ is fibrant by observing that any map
\[
(I\times \partial \Delta[n])\cup(\{v\}\times \Delta[n])\to \Delta[2]^{\ast}_{0\to 1}
\]
either factors through the collapse map $(I\times \partial \Delta[n])\cup(\{v\}\times \Delta[n])\to \Delta[n]$ or factors through $\Delta[1]\hookrightarrow \Delta[2]^{\ast}_{0\to 1}$ because the 1-simplices of $I$ go in opposite directions.
\end{proof}

In Figure \ref{fig:MSwith2seg}, we indicate how we expect the 2-Segal analogue of these model structures to fit in.
\begin{figure}[h]
\caption{}
\label{fig:MSwith2seg}
\vspace{3mm}
\begin{tikzcd}[column sep=small]
\Min(\varnothing) \arrow[rr] \arrow[d] &  & \Min(\text{HtpyBehaved}) \arrow[d]             &  &                                    &  &                       \\
\Min(K'\to \ast) \arrow[rr] \arrow[d]  &  & \Min(K',\text{HtpyBehaved}) \arrow[d] \arrow[rr] &  & \text{2-Special} \arrow[d] \arrow[rr] &  & \text{Quasi-2-Seg} \arrow[d] \\
\Min(K\to \ast) \arrow[rr]             &  & \Min(K,\text{HtpyBehaved}) \arrow[rr]            &  & \text{Special} \arrow[rr] & & \text{Joyal}                
\end{tikzcd}
\end{figure}
In particular, it appears likely that there is a separate simplicial set $K'$ that plays the same role for the 2-Segal situation as $K$ does for quasi-categories. It also seems likely that there be a notion of 2-Segal pre-isomorphism that is distinct from categorical pre-isomorphisms, and hence a notion of 2-special horns that is distinct from that of special horns. These model structures are the topic of future work.

\end{document}